\documentclass[10pt]{article}%
\usepackage{makeidx}
\usepackage{amsfonts}
\usepackage{graphicx}
\usepackage{amsmath}
\usepackage{amssymb}%
\providecommand{\U}[1]{\protect\rule{.1in}{.1in}}
\newtheorem{theorem}{Theorem}

\newtheorem{corollary}[theorem]{Corollary}

\newtheorem{definition}[theorem]{Definition}

\newtheorem{lemma}[theorem]{Lemma}

\newtheorem{proposition}[theorem]{Proposition}
\newtheorem{remark}[theorem]{Remark}

\newenvironment{proof}[1][Proof]{\textbf{#1.} }{\ \rule{0.5em}{0.5em}}
\begin{document}

\title{Error bounds for interpolation with piecewise exponential splines
of order two and four}
\author{O. Kounchev, H. Render}
\maketitle

\begin{abstract}
Explicit pointwise error bounds for the interpolation of a smooth function by
piecewise exponential splines of order four are given. Estimates known for
cubic splines are extended to a natural class of piecewise exponential splines
which are appearing in the construction of multivariate polysplines. The error
estimates are derived in an inductive way using error estimates for the
interpolation of a smooth function by exponential splines of order two.


\medskip\textbf{Acknowledgement:} Both authors were supported by grant DH
02-13 and grant KP-06-N32-8 with Bulgarian NSF, while the first was partially
supported by the Alexander von Humboldt Foundation, Bonn.

\end{abstract}

Let $C^{m}\left[  a,b\right]  $ be the space of all $m$ times continuously
differentiable functions $f:\left[  a,b\right]  \rightarrow\mathbb{C}$ on the
closed interval $\left[  a,b\right]  $. A function $g:\left[  t_{1}%
,t_{n}\right]  \rightarrow\mathbb{C}$ is an \emph{exponential spline}%
\footnote{We follow here the terminology in \cite[p. 405]{Sc81}. In
\cite{McCa90}, \cite{McCa91} tensions splines with varying parameters are
called exponential splines. This terminology seems to be misleading.}
for the knots $t_{1}<...<t_{n}$ and $\Lambda=(\lambda_{0},....,\lambda_{N}%
)\in\mathbb{C}^{N+1}$ if $g\in C^{N-1}\left[  t_{1},t_{n}\right]  $ and the
restriction of $g$ to each open interval $\left(  t_{j},t_{j+1}\right)  $ is a
solution of the differential equation $L_{\left(  \lambda_{0},...,\lambda
_{N}\right)  }\left(  g\right)  =0$ where
\begin{equation}
L_{\left(  \lambda_{0},...,\lambda_{N}\right)  }=\prod_{j=0}^{N}\left(
\frac{d}{dx}-\lambda_{j}\right)  . \label{neuLL}%
\end{equation}
Exponential splines are used in many applications, e.g. in signal processing
\cite{Unser2005}, in non-parametric regression \cite{GreenSilverman},
\cite{RamsaySilverman}, in statistical modelling and smoothing of big data,
see \cite{Gu}, \cite{HastieTibshiraniFriedman}, \cite{Wahba1990}, and for
approximating solutions of partial differential equations (see \cite{Moha14}
and its references). The choice of the differential operator in (\ref{neuLL})
depends on the specific aspects of the underlying problem. The motivation for
this paper came from the effort to provide error estimates for polysplines of
order 4, \cite{KounchevBOOK}. Our main interest is related to the differential
operator
\begin{equation}
L_{\left(  \xi,\xi,-\xi,-\xi\right)  }=\left(  \frac{d^{2}}{dt^{2}}-\xi
^{2}\right)  ^{2}\text{ with }\xi\in\mathbb{R} \label{eqsymmetricF}%
\end{equation}
which arises naturally in the context of biharmonic functions, see
\cite{KoRe13}, \cite{KoReMathNach} and \cite{KoReInter2019}. Polysplines on
strips can be described by Fourier methods and the parameter $\xi$ in
(\ref{eqsymmetricF}) is equal to $\left\vert y\right\vert $ where
$y\in\mathbb{R}^{n}$ is arbitrary, see \cite{KounchevBOOK},
\cite{kounchevrenderappr}, \cite{kounchevrenderpams}, \cite{KoRe05},
\cite{kounchevrenderJAT}. Hence it is crucial to have error estimates which
are valid uniformly for all parameters $\xi\in\mathbb{R}.$

Let us emphasize that $L^{\infty}$- and $L^{2}$-error estimates for (more
general) interpolation $L$-splines have been discussed by several authors, see
\cite{ScSc80}, \cite{ScVa67}, \cite{Sc81}. Our aim is to achieve exact control
of the constants in the estimates which depend on the differential operator,
in particular we want to obtain error estimates in the case of operator
(\ref{eqsymmetricF}) uniformly for all $\xi\in\mathbb{R}.$

For $F\in C^{1}\left[  t_{1},t_{n}\right]  $ let us denote by $I_{4}\left(
F\right)  $ the exponential spline for the operator $L_{\left(  \lambda
_{0},\lambda_{1},\lambda_{2},\lambda_{3}\right)  }$ satisfying the conditions
\begin{align}
I_{4}\left(  F\right)  \left(  t_{j}\right)   &  =F\left(  t_{j}\right)
\text{ for }j=1,...,n,\text{ }\label{eqInterpolB1}\\
\frac{d}{dt}I_{4}\left(  F\right)  \left(  t_{1}\right)   &  =\frac{d}%
{dt}F\left(  t_{1}\right)  \text{ and }\frac{d}{dt}I_{4}\left(  F\right)
\left(  t_{n}\right)  =\frac{d}{dt}F\left(  t_{n}\right)  .
\label{eqInterpolB2}%
\end{align}
We want to determine \emph{explicit constants} $C=C_{\left(  \lambda
_{0},\lambda_{1},\lambda_{2},\lambda_{3}\right)  }$ such that for any $F\in
C^{4}\left[  t_{1},t_{n}\right]  $ and any partition $t_{1}<...<t_{n}$ the
following error estimate
\begin{equation}
\left\vert F\left(  t\right)  -I_{4}\left(  F\right)  \left(  t\right)
\right\vert \leq C\cdot\Delta^{4}\cdot\max_{x\in\left[  t_{1},t_{n}\right]
}\left\vert L_{\left(  \lambda_{0},\lambda_{1},\lambda_{2},\lambda_{3}\right)
}F\left(  x\right)  \right\vert \label{eqError}%
\end{equation}
holds for $t\in\left[  t_{1},t_{n}\right]  $ where
\begin{equation}
\Delta:=\max_{j=1,...,n-1}\left\vert t_{j+1}-t_{j}\right\vert .
\label{defDelta}%
\end{equation}
In order to a give a flavor of the results in this paper we state now one of
the major results:

\begin{theorem}
\label{MainThmIntro}For $t_{1}<...<t_{n}$ and $F\in C^{4}\left[  t_{1}%
,t_{n}\right]  $ let $I_{4}\left(  F\right)  $ be the exponential spline for
the operator $L_{\left(  \xi,\xi,-\xi,-\xi\right)  }$ satisfying
(\ref{eqInterpolB1}) and (\ref{eqInterpolB2}). Then
\begin{equation}
\max_{t\in\left[  t_{1},t_{n}\right]  }\left\vert F\left(  t\right)
-I_{4}\left(  F\right)  \left(  t\right)  \right\vert \leq\Delta^{4}\cdot
\frac{5}{64}\ \max_{\theta\in\left[  t_{1},t_{n}\right]  }\left\vert
L_{\left(  \xi,\xi,-\xi,-\xi\right)  }F\left(  \theta\right)  \right\vert .
\label{eqErrorsym}%
\end{equation}

\end{theorem}

When we put $\xi=0$ in (\ref{eqErrorsym}) we obtain an error estimate for
cubic splines. Note that then our result provides the estimate for cubic
splines\footnote{As already mentioned in \cite{deBoor}, Hall and Meyer have
shown in \cite{HaMe76} that the best constant $C$ in (\ref{eqErrorsym}) for
cubic splines is $5/384$.} in \cite[p. 55]{deBoor} but with an additional
factor $5/4$. However, our estimate holds uniformly for all $\xi\in
\mathbb{R}.$

Our approach to the error estimates for exponential splines of order $4$ is
inspired by the elegant exposition of Carl de Boor in \cite{deBoor} of the
error estimates for interpolation cubic splines. It is shown there that this
error can be estimated in two steps by using two simpler error estimates: the
pointwise error estimate for \emph{interpolation} \emph{with continuous
piecewise linear functions} for the knots $t_{1}<...<t_{n}$ (which is rather
easy) and the pointwise error estimate of \emph{best }$L^{2}$%
\emph{-approximation by continuous piecewise linear functions} for
$t_{1}<...<t_{n} $ (which is more delicate). For the latter estimate diagonal
dominance of certain matrices is an important tool, and, according to \cite[p.
35]{deBoor}, this technique in spline analysis was first used in
\cite{ShMe66}. In the case of exponential splines we shall follow the same
ideas and prove generalizations of every step. According to our knowledge this
is a new contribution in the literature. In any case, the computational
aspects are lot more challenging than in the polynomial case.

Let us now outline the structure and the main results of the paper: in section
$2$ we introduce at first the concept of generalized hat functions. In section
$3$ exponential splines of order $2$ for the differential operator $L_{\left(
\lambda_{0},\lambda_{1}\right)  }$ are discussed which interpolate a function
$f$ at the points $t_{1},...,t_{n}$ (the analog of interpolating linear
splines). This interpolation exponential spline of order $2$ will be denoted
by $I_{2}\left(  f\right)  $. In section $4$ we provide the following error
estimate
\[
\left\vert f\left(  t\right)  -I_{2}\left(  f\right)  \left(  t\right)
\right\vert \leq\max_{j=1,...,n-1}M_{\lambda_{0},\lambda_{1}}^{t_{j},t_{j+1}%
}\max_{\theta\in\left[  t_{1},t_{n}\right]  }\left\vert L_{\left(  \lambda
_{0},\lambda_{1}\right)  }f\left(  \theta\right)  \right\vert
\]
for all $t\in\left[  t_{1},t_{n}\right]  .$ Here the constant $M_{\lambda
_{0},\lambda_{1}}^{a,b}$ for the interval $\left[  a,b\right]  $ is defined
by
\[
M_{\lambda_{0},\lambda_{1}}^{a,b}=\max_{t\in\left[  a,b\right]  }\left\vert
\Omega_{\lambda_{0},\lambda_{1},0}^{a,b}\left(  t\right)  \right\vert ,
\]
and $\Omega_{\lambda_{0},\lambda_{1},0}^{a,b}$ is a solution of the
differential equation $L_{\left(  \lambda_{0},\lambda_{1},0\right)  }u=0$ such
that
\[
\Omega_{\lambda_{0},\lambda_{1},0}^{a,b}\left(  a\right)  =\Omega_{\lambda
_{0},\lambda_{1},0}^{a,b}\left(  b\right)  =0\text{ and }L_{\left(
\lambda_{0},\lambda_{1}\right)  }\Omega_{\lambda_{0},\lambda_{1},0}^{a,b}=-1.
\]
In section $5$ we shall compute the constant $M_{\lambda_{0},\lambda_{1}%
}^{a,b}$ for several cases. For the operator $L_{\left(  \xi,-\xi\right)  }$
we shall prove the estimate
\[
M_{\xi,-\xi}^{a,b}\leq\frac{1}{8}\left(  b-a\right)  ^{2}%
\]
for all $a<b.$ It is surprising that in the case $\lambda_{0}\leq0\leq
\lambda_{1}$ the following estimate holds:
\[
M_{\lambda_{0},\lambda_{1}}^{a,b}<\frac{1}{4}\left(  b-a\right)  ^{2}.
\]
Simple examples show that for positive $\lambda_{0},\lambda_{1}$ the numbers
$M_{\lambda_{0},\lambda_{1}}^{a,b}/\left(  b-a\right)  ^{2}$ are not bounded.

Section $5$ is devoted to an error estimate of the \emph{best }$L^{2}%
$\emph{-approximation to }$f\in C^{2}\left[  a,b\right]  $\emph{\ by
exponential splines of order} $2$ (for the operator $L_{\left(  \lambda
_{0},\lambda_{1}\right)  })$ for the knots $t_{1}<...<t_{n}.$ Let us denote by
$\mathcal{H}_{n}$ the space of exponential splines of order $2$ with respect
to $L_{\left(  \lambda_{0},\lambda_{1}\right)  }$ and for the partition
$t_{1}<...<t_{n}.$ We denote by $P^{\mathcal{H}_{n}}\left(  f\right)  $ the
best $L_{2}$-approximation to $f$ from the subspace $\mathcal{H}_{n}.$
Following the pathway provided in \cite{deBoor} (estimate of the best
approximation by hat functions and estimate of solutions of tridiagonal
matrices) one obtains explicit estimates for the operator norm $\left\Vert
P^{\mathcal{H}_{n}}\right\Vert _{\text{op}}$ (defined with respect to the
uniform norm), see Theorem \ref{ThmMain1}. For the important case $\lambda
_{0}=\xi$ and $\lambda_{1}=-\xi$ for $\xi\in\mathbb{R}$ we show that%
\[
\left\Vert P^{\mathcal{H}_{n}}\right\Vert _{\text{op}}\leq4,
\]
which is slightly larger than the constant in the polynomial case which is
$3,$ see \cite[p. 34]{deBoor}.

In section $6$ we use the results of error estimates of exponential splines of
order $2$ for a short proof of the error estimate for exponential splines for
differential operators of the type $L_{\left(  \lambda_{0},\lambda
_{1},-\lambda_{0},-\lambda_{1}\right)  }$.

In the Appendix, section $7,$ we have compiled some results about exponential
polynomials which are needed in the previous sections.

Most of the above results are valid for the so-called \emph{piecewise}
exponential splines. Let us recall that $g:\left[  t_{1},t_{n}\right]
\rightarrow\mathbb{C}$ is a \emph{piecewise exponential spline} for the knots
$t_{1}<...<t_{n}$ and variable frequencies $\left(  \lambda_{0,j}%
,....,\lambda_{N,j}\right)  $ for $j=1,...,n-1,$ if $g\in C^{N-1}\left[
t_{1},t_{n}\right]  $ and the restriction of $g$ to each interval $\left(
t_{j},t_{j+1}\right)  $ is a solution of the equation $L_{\left(
\lambda_{0,j},...,\lambda_{N,j}\right)  }\left(  g\right)  =0$ depending on
$j=1,...,n-1.$ This concept was introduced by Sp\"{a}th in \cite{Spaeth} in
the context of tension splines where different values of the tension parameter
$\rho_{j}^{2}$ could be chosen for the intervals $\left(  t_{j-1}%
,t_{j}\right)  .$ Recall that \emph{splines in tension} are exponential
splines for the differential operator
\begin{equation}
L_{\left(  0,0,\rho,-\rho\right)  }=\frac{d^{2}}{dt^{2}}\left(  \frac{d^{2}%
}{dt^{2}}-\rho^{2}\right)  , \label{eqtensionOp}%
\end{equation}
see also \cite{Prue76}, \cite{Prue78}, \cite{Prue79}. The extension of the
above results to \emph{piecewise} \emph{L-splines} does require only one
additional burden -- that for terminology and notation. Let us mention that
some results in the literature about piecewise L-splines are erroneous (e.g.
in \cite{Prenter}) as pointed out recently by Z. Ayalon, N. Dyn and D. Levin
in \cite{ADL09}.

Finally let us introduce some standard notations. We denote the space of all
\emph{exponential polynomials} or $L$\emph{-polynomials }with respect to the
differential operator $L_{\left(  \lambda_{0},\ldots,\lambda_{N}\right)  }$
by
\[
E\left(  \lambda_{0},...,\lambda_{N}\right)  =\left\{  f\in C^{N+1}\left(
\mathbb{R}\right)  :L_{\left(  \lambda_{0},\ldots,\lambda_{N}\right)
}f=0\right\}
\]
The maximum norm for $f\in C\left[  a,b\right]  $ is defined by%
\[
\left\Vert f\right\Vert _{\left[  a,b\right]  }:=\max_{t\in\left[  a,b\right]
}\left\vert f\left(  t\right)  \right\vert
\]
and the distance of $f\in C\left[  a,b\right]  $ to a subspace $U$ of
$C\left[  a,b\right]  $ is defined by
\[
\text{dist}_{C\left[  a,b\right]  }\left(  f,U\right)  =\inf\left\{
\left\Vert f-g\right\Vert _{\left[  a,b\right]  }:g\in U\right\}  .
\]

\section{\label{S2}Generalized hat functions}

Hat functions, also called chapeau functions, are used in finite element
methods. It is well known that hat functions provide a basis of the space of
linear splines, cf. \cite{deBoor}. From the viewpoint of spline analysis hat
functions are just linear splines with minimal support, or briefly, linear B-splines.

Now we shall generalize this concept, see Definition \ref{DefHat} below. When
$\varphi_{j}\left(  t\right)  =t$ for $t\in\mathbb{R}$ and $j=1,...,n-1, $ we
obtain the definition of hat functions $H_{1},....,H_{n}$ as given in \cite[p.
32]{deBoor}. In many applications we take the same function $\varphi
=\varphi_{j}$ for $j=1,...,n-1$. Our definition allows us to develop error
estimates for the case of \emph{piecewise} exponential splines.

\begin{definition}
\label{DefHat}Let $\delta>0$, $n\geq3$ and let $\varphi_{j}:\left[
-\delta,\delta\right]  \rightarrow\mathbb{C}$ be continuous strictly
increasing functions with $\varphi_{j}\left(  0\right)  =0$ for $j=1,...,n-1.$
Let the points $t_{1}<...<t_{n}$ be given such that $t_{j}-t_{j-1}\leq\delta$
for all $j=2,...,n.$ Define for $j=2,...,n-1$ the hat function $H_{j}\left(
t\right)  $ with support in $\left[  t_{j-1},t_{j+1}\right]  $ by%
\[
H_{j}\left(  t\right)  =\left\{
\begin{array}
[c]{ccc}%
\frac{\varphi_{j-1}\left(  t-t_{j-1}\right)  }{\varphi_{j-1}\left(
t_{j}-t_{j-1}\right)  } &  & \text{for }t\in\left[  t_{j-1},t_{j}\right] \\
\frac{\varphi_{j}\left(  t-t_{j+1}\right)  }{\varphi_{j}\left(  t_{j}%
-t_{j+1}\right)  } &  & \text{for }t\in\left[  t_{j},t_{j+1}\right]  .
\end{array}
\right.
\]
Further for $j=1$ and $j=n$ we define
\begin{align*}
H_{1}\left(  t\right)   &  =\frac{\varphi_{1}\left(  t-t_{2}\right)  }%
{\varphi_{1}\left(  t_{1}-t_{2}\right)  }\text{ for }t\in\left[  t_{1}%
,t_{2}\right]  ,\\
H_{n}\left(  t\right)   &  =\frac{\varphi_{n-1}\left(  t-t_{n-1}\right)
}{\varphi_{n-1}\left(  t_{n}-t_{n-1}\right)  }\text{ for }t\in\left[
t_{n-1},t_{n}\right]  ,
\end{align*}
and zero elsewhere.
\end{definition}

Obviously $H_{1},...,H_{n}$ are non-negative and linearly independent. Given a
function $f\in C\left[  t_{1},t_{n}\right]  $ we define the \emph{interpolant
with respect to the hat functions} $H_{1},...,H_{n}$ as
\[
I_{2}\left(  f\right)  =\sum_{j=1}^{n}f\left(  t_{j}\right)  H_{j}.
\]
Note that $I_{2}\left(  f\right)  \left(  t_{j}\right)  =f\left(
t_{j}\right)  $ for $j=1,...,n,$ so $I_{2}\left(  f\right)  $ interpolates $f$
at the points $t_{1},...,t_{n}.$

\begin{proposition}
\label{PropHatBasics}Assume that $\varphi_{j}$ are strictly increasing
functions on $\left[  -\delta,\delta\right]  $ such that $\varphi_{j}\left(
0\right)  =0$ for each $j=1,...,n-1,$ and $\max_{j=2,...,n}\left(
t_{j}-t_{j-1}\right)  \leq\delta.$ Then
\begin{equation}
\sum_{j=1}^{n}\left\vert H_{j}\left(  t\right)  \right\vert \leq2\text{ and
}\left\vert I_{2}\left(  f\right)  \left(  t\right)  \right\vert
\leq2\left\Vert f\right\Vert _{\left[  t_{1},t_{n}\right]  } \label{eqestHj}%
\end{equation}
for $f\in C\left[  t_{1},t_{n}\right]  $ and $t\in\left[  t_{1},t_{n}\right]
.$ If $U_{n}$ denotes the linear space generated by $H_{1},...,H_{n}$ then
\begin{equation}
\text{dist}_{C\left[  a,b\right]  }\left(  f,U_{n}\right)  \leq\left\Vert
f-I_{2}\left(  f\right)  \right\Vert _{\left[  t_{1},t_{n}\right]  }\leq
3\cdot\text{dist}_{C\left[  a,b\right]  }\left(  f,U_{n}\right)
\label{bestapprox}%
\end{equation}

\end{proposition}

\begin{proof}
For $j=2,...,n$ it is obvious that $H_{j}\left(  t\right)  =\frac
{\varphi_{j-1}\left(  t-t_{j-1}\right)  }{\varphi_{j-1}\left(  t_{j}%
-t_{j-1}\right)  }$ is increasing on $\left[  t_{j-1},t_{j}\right]  .$ Note
that $\varphi_{j}\left(  t_{j}-t_{j+1}\right)  $ is negative and
$t\longmapsto\varphi_{j}\left(  t-t_{j+1}\right)  $ is increasing, hence
$H_{j}\left(  t\right)  =\frac{\varphi_{j}\left(  t-t_{j+1}\right)  }%
{\varphi_{j}\left(  t_{j}-t_{j+1}\right)  }$ is decreasing on $\left[
t_{j},t_{j+1}\right]  $ for $j=1,...n-1.$ It is now easy to see that%
\[
0\leq H_{j}\left(  t\right)  \leq1\text{ for all }t\in\left[  t_{1}%
,t_{n}\right]
\]
and $j=1,...,n.$ Since each $H_{j}$ has support in $\left[  t_{j-1}%
,t_{j+1}\right]  $ the statement in (\ref{eqestHj}) is now obvious. Note that
the first inequality in (\ref{bestapprox}) is trivial. For $g\in U_{n}$ we
have $g=I_{2}\left(  g\right)  $ and
\[
\left\Vert f-I_{2}\left(  f\right)  \right\Vert _{\left[  t_{1},t_{n}\right]
}\leq\left\Vert f-g\right\Vert _{\left[  t_{1},t_{n}\right]  }+\left\Vert
I_{2}\left(  g-f\right)  \right\Vert _{\left[  t_{1},t_{n}\right]  }%
\leq3\left\Vert f-g\right\Vert _{\left[  t_{1},t_{n}\right]  }.
\]
Taking the infimum over all $g\in U_{n}\mathcal{\ }$gives the result.
\end{proof}

\section{\label{S3}Piecewise exponential splines of order $2$ and generalized
hat functions}

For arbitrary complex numbers $\lambda_{0},\lambda_{1}$ we define
$\Phi_{\left(  \lambda_{0},\lambda_{1}\right)  }$ to be the unique exponential
polynomial in $E\left(  \lambda_{0},\lambda_{1}\right)  $ satisfying
$\Phi_{\left(  \lambda_{0},\lambda_{1}\right)  }\left(  0\right)  =0$ and
$\Phi_{\left(  \lambda_{0},\lambda_{1}\right)  }^{\prime}\left(  0\right)
=1$. We call $\Phi_{\left(  \lambda_{0},\lambda_{1}\right)  }$ the fundamental
function for $\left(  \lambda_{0},\lambda_{1}\right)  $, see the appendix for
a general discussion. For $\lambda_{0}\neq\lambda_{1}$ the simple formula
\[
\Phi_{\left(  \lambda_{0},\lambda_{1}\right)  }\left(  t\right)
=\frac{e^{\lambda_{1}t}-e^{\lambda_{0}t}}{\lambda_{1}-\lambda_{0}}=e^{\left(
\lambda_{0}+\lambda_{1}\right)  t/2}\frac{e^{\left(  \lambda_{1}-\lambda
_{0}\right)  t/2}-e^{-\left(  \lambda_{1}-\lambda_{0}\right)  t/2}}%
{\lambda_{1}-\lambda_{0}}.
\]
holds. In the case that $\lambda_{0}=\lambda_{1}\neq0$ we define
$\Phi_{\left(  \lambda_{0},\lambda_{0}\right)  }\left(  t\right)
=e^{\lambda_{0}t}\frac{t}{\lambda_{0}}.$ In the case $\lambda_{0}=\lambda
_{1}=0$ we just have $\Phi_{\left(  0,0\right)  }\left(  t\right)  =t.$ In any
of these three cases there exists an odd function $\psi_{\left(  \lambda
_{0},\lambda_{1}\right)  }$ such that $\Phi_{\left(  \lambda_{0},\lambda
_{1}\right)  }\left(  t\right)  =e^{\left(  \lambda_{0}+\lambda_{1}\right)
t/2}\psi_{\left(  \lambda_{0},\lambda_{1}\right)  }\left(  t\right)  ,$ and
this leads to the following useful formula:
\begin{equation}
\Phi_{\left(  \lambda_{0},\lambda_{1}\right)  }\left(  -t\right)
=-e^{-\left(  \lambda_{0}+\lambda_{1}\right)  t}\Phi_{\left(  \lambda
_{0},\lambda_{1}\right)  }\left(  t\right)  . \label{eqsym}%
\end{equation}
For simplicity we shall assume in this paper that $\lambda_{0},\lambda_{1}$
are real numbers although it might be interesting for applications to include
exponential polynomials with complex frequencies (see e.g. \cite{Unser2005}):
if $\lambda_{0}=i\alpha$ and $\lambda_{1}=-i\alpha$ for $\alpha>0$ then
\[
\Phi_{\left(  i\alpha,-i\alpha\right)  }\left(  t\right)  =\frac{1}{\alpha
}\sin\alpha t
\]
is real-valued and strictly increasing on $\left[  -\frac{\pi}{2\alpha}%
,\frac{\pi}{2\alpha}\right]  .$

The following example on Figure 1 shows that the function $\Phi_{\left(  \lambda
_{0},\lambda_{1}\right)  }$ is not always increasing on the real line even if
$\lambda_{0}$ and $\lambda_{1}$ are real:

%

\begin{figure}
[h]
\begin{center}
\includegraphics[
height=2.1423in,
width=4.1945in
]%
{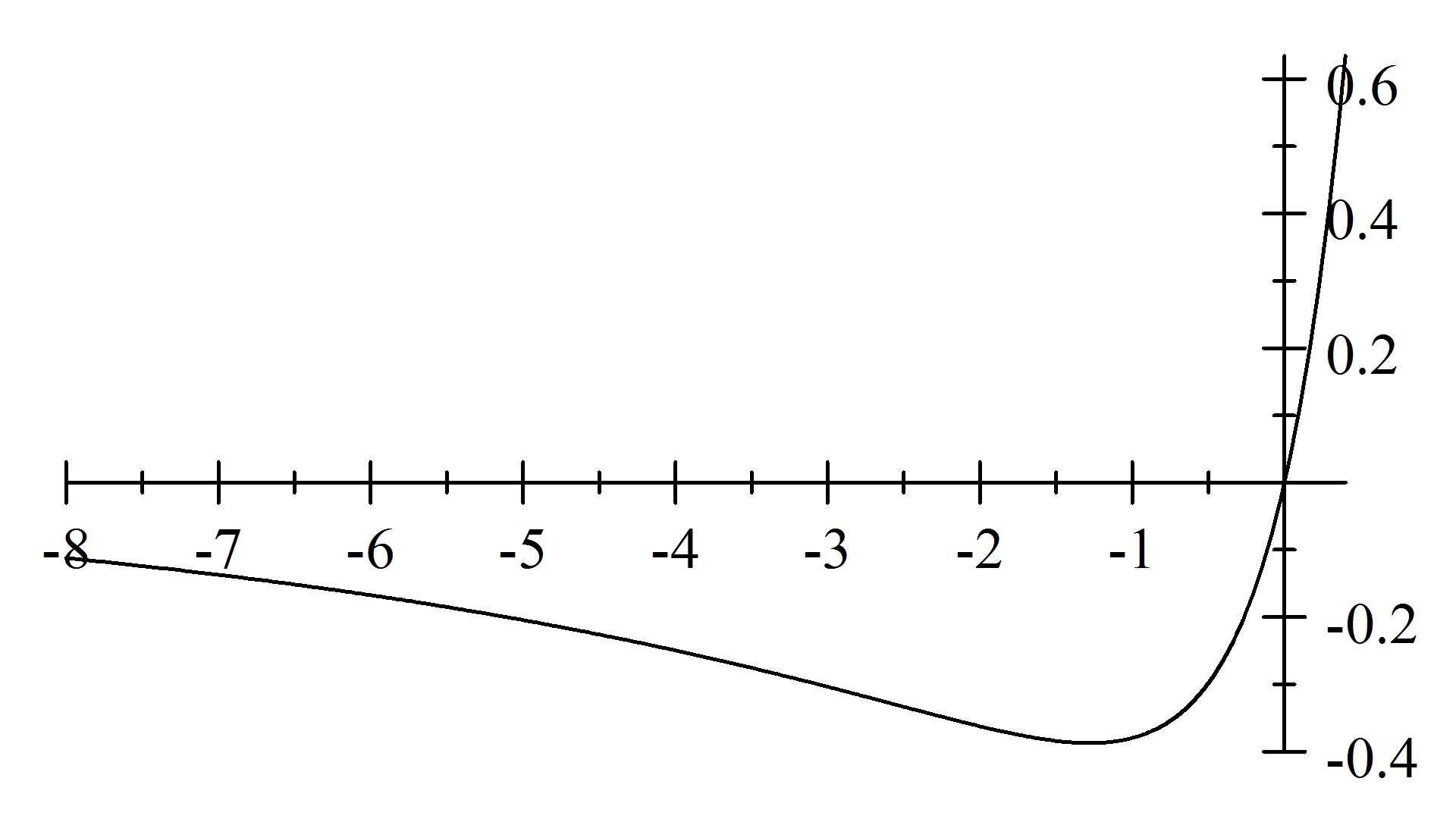}%
\caption{Graph of the fundamental function}%
\end{center}
\end{figure}


\begin{proposition}
\label{PropIncreasing}Assume that $\lambda_{0},\lambda_{1}$ are real numbers.
Then the following statements hold:

(i) If $\lambda_{0}\leq0\leq\lambda_{1}$ then $\Phi_{\left(  \lambda
_{0},\lambda_{1}\right)  }$ is strictly increasing on $\mathbb{R}.$

(ii) If $0<\lambda_{0}<\lambda_{1}$ then $\Phi_{\left(  \lambda_{0}%
,\lambda_{1}\right)  }$ is strictly increasing on $\left[  -\delta
,\delta\right]  $ with $\delta=\frac{\ln\lambda_{1}-\ln\lambda_{0}}%
{\lambda_{1}-\lambda_{0}}.$
\end{proposition}

\begin{proof}
The case $\lambda_{0}=\lambda_{1}$ leads in (i) to the polynomial case
$\lambda_{0}=\lambda_{1}=0.$ For $\lambda_{0}\leq0\leq\lambda_{1}$ with
$\lambda_{0}\neq\lambda_{1}$ the statement is obvious since
\[
\frac{d}{dt}\Phi_{\left(  \lambda_{0},\lambda_{1}\right)  }\left(  t\right)
=\frac{\lambda_{1}e^{\lambda_{1}t}-\lambda_{0}e^{\lambda_{0}t}}{\lambda
_{1}-\lambda_{0}}\geq0.
\]
In the case (ii), $\Phi_{\left(  \lambda_{0},\lambda_{1}\right)  }\left(
t\right)  $ has one critical point $t_{0}$, namely
\[
e^{\left(  \lambda_{1}-\lambda_{0}\right)  t_{0}}=\frac{\lambda_{0}}%
{\lambda_{1}}\text{, so }t_{0}=\frac{\ln\lambda_{0}-\ln\lambda_{1}}%
{\lambda_{1}-\lambda_{0}}.
\]
The function $t\longmapsto\Phi_{\left(  \lambda_{0},\lambda_{1}\right)
}\left(  t\right)  $ is increasing for all $t>t_{0}$, and decreasing for all
$t<t_{0}$. Note that $\Phi_{\left(  \lambda_{0},\lambda_{1}\right)  }\left(
t\right)  \rightarrow0$ for $t\rightarrow-\infty.$
\end{proof}

\bigskip

Throughout the whole paper we make the\textbf{\ }\emph{following assumptions}
and use of notations:

\begin{enumerate}
\item[(i)] Let $\delta>0.$ We assume that for the real numbers $\lambda
_{0,j}\leq\lambda_{1,j}$ the functions%
\begin{equation}
\varphi_{j}:=\Phi_{\left(  \lambda_{0,j},\lambda_{1,j}\right)  },\quad\quad
j=1,...,n-1,\label{fi=FI}%
\end{equation}
are increasing on $\left[  -\delta,\delta\right]  .$

\item[(ii)] For given $t_{1}<...<t_{n}$ with $\left\vert t_{j+1}%
-t_{j}\right\vert \leq\delta,$ for $j=1,..,n-1$ the corresponding hat
functions are denoted by $H_{1},....,H_{n}$, and their linear span is denoted
by $\mathcal{H}_{n}$.
\end{enumerate}

The following result is now obvious:

\begin{proposition}
\label{PropDefphij} The generalized hat functions $H_{1},...,H_{n}$ form a
linear basis of the vector space $\mathcal{H}_{n}$ of all piecewise
exponential splines for the knots $t_{1}<...<t_{n}$ and the differential
operators $L_{\left(  \lambda_{0,j},\lambda_{1,j}\right)  }$ on the interval
$\left(  t_{j},t_{j+1}\right)  $ for $j=1,..,n-1.$
\end{proposition}

The generalized hat functions $H_{1},...,H_{n}$ in Proposition
\ref{PropDefphij} are piecewise exponential splines of order $2$ with minimal
compact support, or shorter, \emph{piecewise exponential B-splines of order
}$2$. B-splines for special classes of exponential polynomials have been used
by many authors, see e.g. \cite[p. 197]{Eddy} or \cite{WaFa06}.

One basic feature of linear hat functions is the partition of unity, saying
that the expression
\[
U\left(  t\right)  :=\sum_{j=1}^{n}H_{j}\left(  t\right)
\]
is equal to the constant function $1.$ Note that $U\left(  t\right)  $ is the
piecewise exponential spline of order $2$ interpolating the constant function
$1,$ and in general this function is not constant. The following Figure $2$
shows  the three basis functions $H_{1},$ $H_{2},$ $H_{3}$ for the case
$\lambda_{1}=-\lambda_{0}$:%


\begin{figure}
[h] 
\begin{center}
\includegraphics[
height=2.1423in,
width=4.1945in
]%
{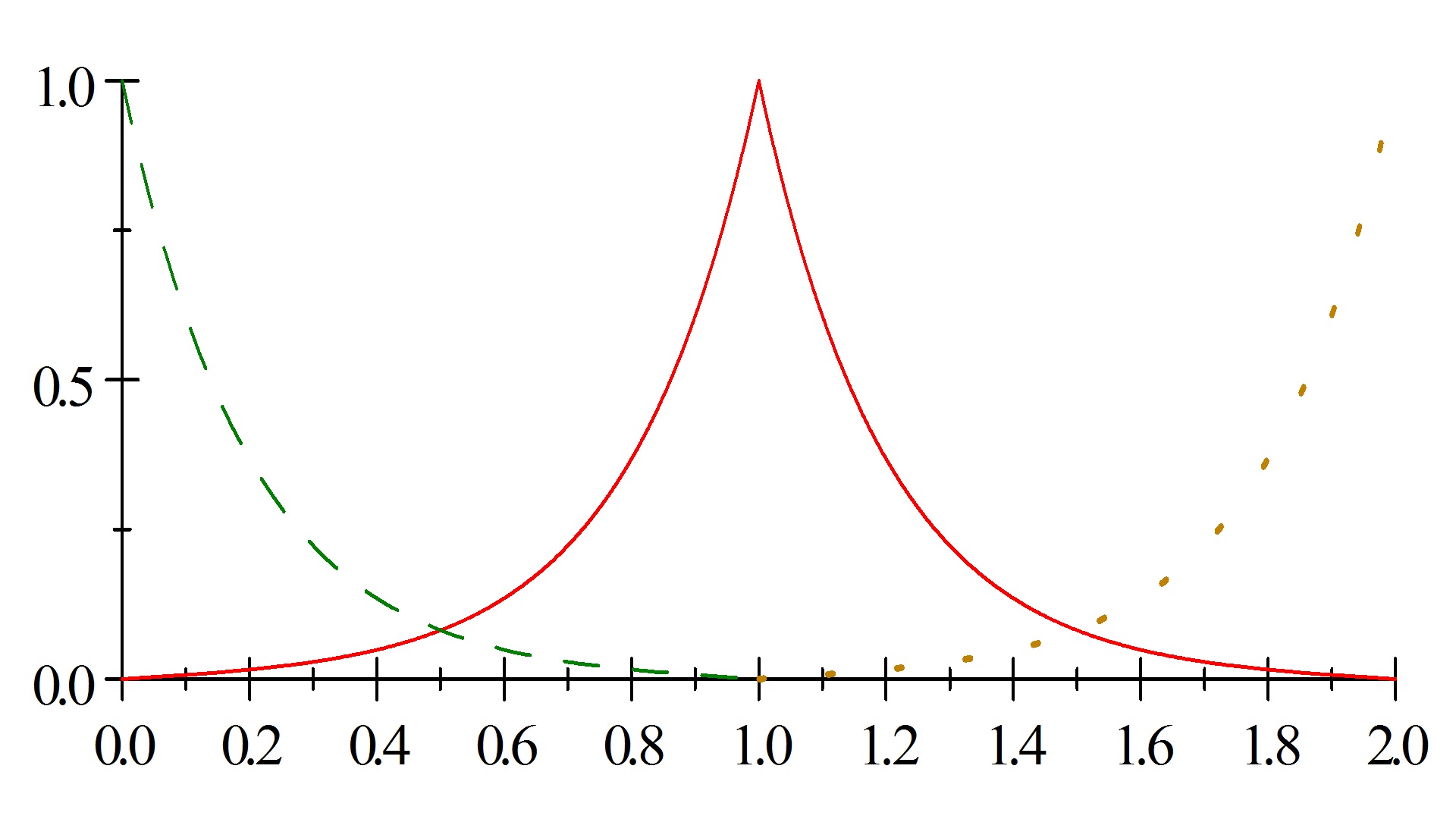}%
\caption{$H_{1},H_{2},H_{3}$ for $\lambda_{0}=-5$ and $\lambda_{1}%
=5$.}%
\end{center}
\end{figure}

On Figure $3$ we see the sum $U\left(  t\right)  $ of the above
basis functions $H_{1}$, $H_{2},$ $H_{3}$. 

%


\begin{figure}
[h] 
\begin{center}
\includegraphics[
height=2.1423in,
width=4.1945in
]%
{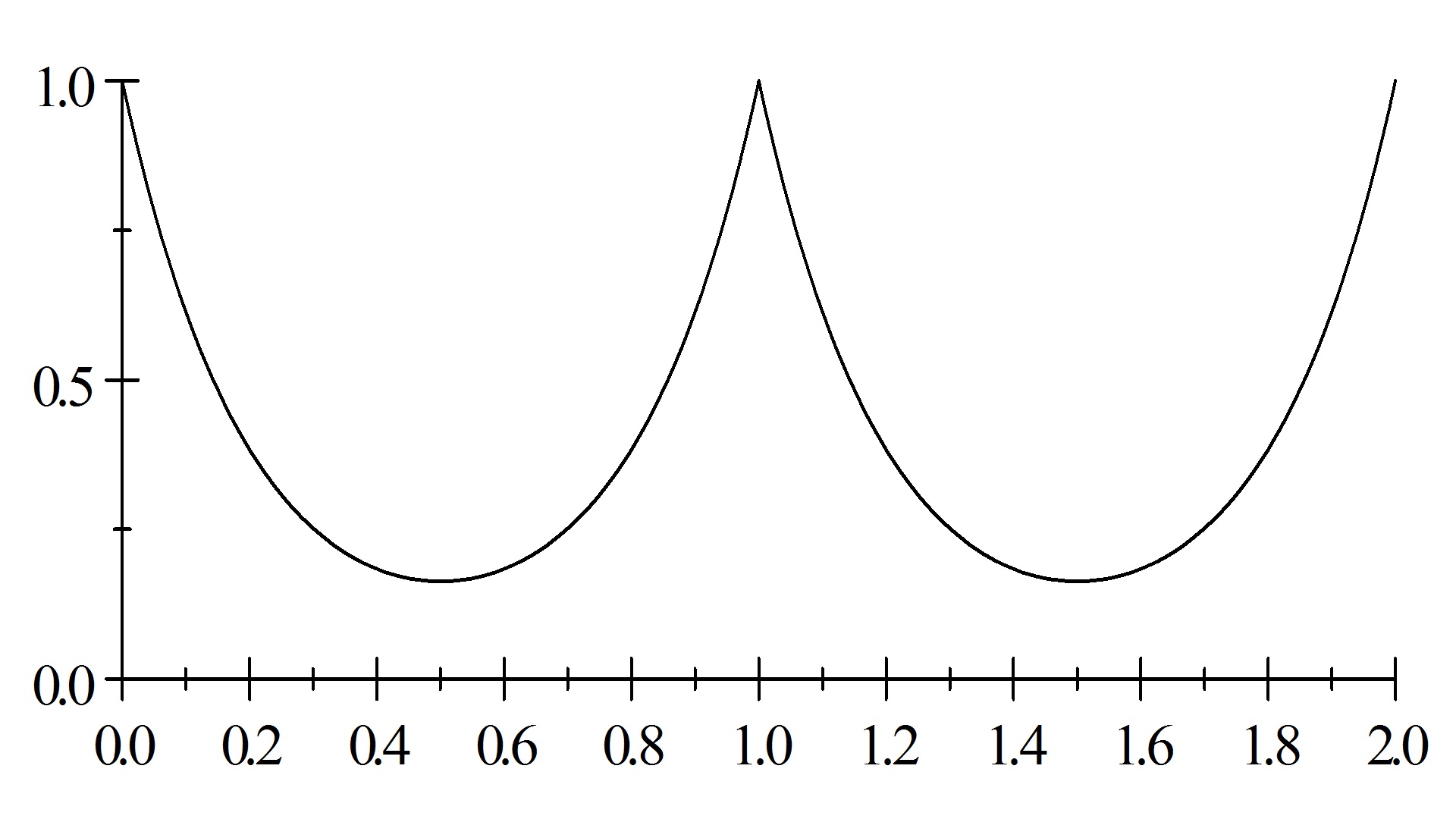}%
\caption{Sum of $H_{1},H_{2},H_{3}$ for $\lambda_{0}=-5$ and
$\lambda_{1}=5$.}%
\end{center}
\end{figure}

For positive frequencies $\lambda_{0}=0.2$ and $\lambda_{1}=2$ by Proposition
\ref{PropIncreasing} we know that $\Phi_{\left(  \lambda_{0},\lambda
_{1}\right)  }$ is increasing on $\left[  -1.2,1.2\right]  $ since $\frac
{\ln2-\ln0.2}{2-0.2}=1.279\,2.$ Figure $4$ shows that in this case the sum of
the hat functions (the upper curve) is not bounded by $1$. 

%


\begin{figure}
[h] 
\begin{center}
\includegraphics[
height=2.1423in,
width=4.1945in
]%
{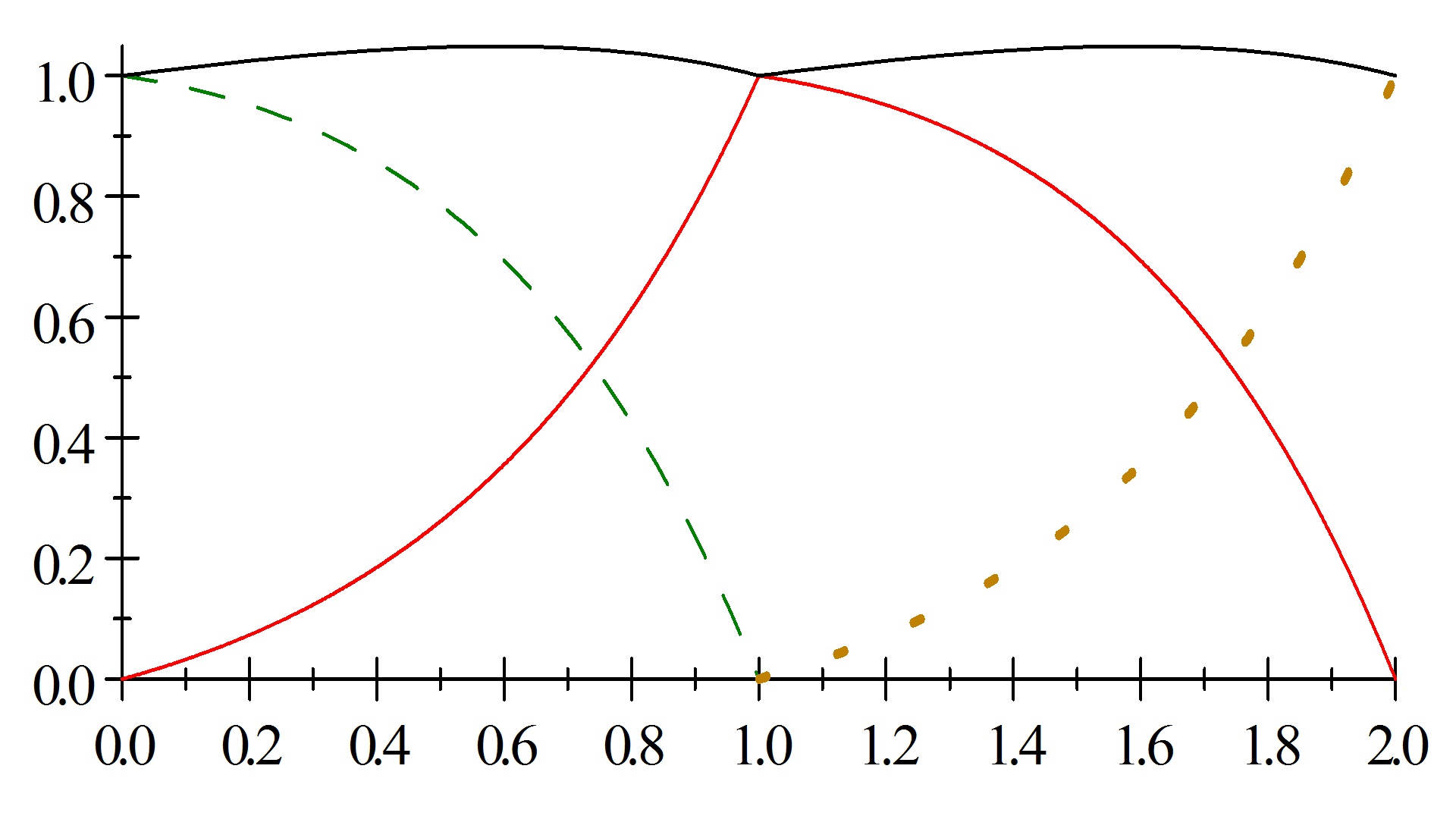}%
\caption{$H_{1},H_{2},H_{3}$ and their sum for $\lambda_{0}=0.2$,
$\lambda_{1}=2$.}%
\end{center}
\end{figure}

In (\ref{eqestHj}) we have seen that $0\leq U\left(  t\right)  \leq2.$ We show
in the next proposition that this inequality can be improved when the
frequencies $\lambda_{0,j}$ and $\lambda_{1,j}$ have different sign for each
$j=1,...,n-1$:

\begin{proposition}
\label{PropSumHat}In addition to (i)--(ii) assume that $\lambda_{0,j}\leq
0\leq\lambda_{1,j}$ for $j=1,...,n-1.$ Then the following inequality holds:
\[
0\leq\sum_{j=1}^{n}H_{j}\left(  t\right)  \leq1.
\]

\end{proposition}

\begin{proof}
It suffices to show that $H_{j}\left(  t\right)  +H_{j+1}\left(  t\right)
\leq1 $ for $t\in\left[  t_{j},t_{j+1}\right]  $ and for $j=1,...,n-1.$ For
$t\in\left[  t_{j},t_{j+1}\right]  $ we have
\[
f\left(  t\right)  :=H_{j}\left(  t\right)  +H_{j+1}\left(  t\right)
=\frac{\varphi_{j}\left(  t-t_{j+1}\right)  }{\varphi_{j}\left(  t_{j}%
-t_{j+1}\right)  }+\frac{\varphi_{j}\left(  t-t_{j}\right)  }{\varphi
_{j}\left(  t_{j+1}-t_{j}\right)  }.
\]
It follows that $f\left(  t_{j}\right)  =f\left(  t_{j+1}\right)  =1$ and
$f\neq0.$ Note that $f$ can not have a local maximum \emph{and} a local
minimum since otherwise the derivative $f^{\prime}$ would have two different
zeros which is a contradiction to the fact that $f^{\prime}\in E\left(
\lambda_{0},\lambda_{1}\right)  $ has at most one zero, see Proposition
\ref{PropCHebyshev} in the Appendix. If we show that $f^{\prime}\left(
t_{j}\right)  <0$ then $f$ will have a local minimum, and not a local maximum,
so $f\left(  t\right)  \leq1$ for $t\in\left[  t_{j},t_{j+1}\right]  , $ and
the proof will be complete.

Writing $\lambda_{0}$ and $\lambda_{1}$ instead of $\lambda_{0,j}$ and
$\lambda_{1,j}$, and $t_{j}=a$ and $t_{j+1}=b,$ by our basic assumption
(\ref{fi=FI}) we have $\varphi_{j}=\Phi_{\left(  \lambda_{0,j},\lambda
_{1,j}\right)  }$, and after putting $h:=t_{j+1}-t_{j},$ we obtain the
equality
\[
f^{\prime}\left(  a\right)  =\frac{\Phi_{\left(  \lambda_{0},\lambda
_{1}\right)  }^{\prime}\left(  -h\right)  }{\Phi_{\left(  \lambda_{0}%
,\lambda_{1}\right)  }\left(  -h\right)  }+\frac{1}{\Phi_{\left(  \lambda
_{0},\lambda_{1}\right)  }\left(  h\right)  }=-e^{\left(  \lambda_{0}%
+\lambda_{1}\right)  h}\frac{\Phi_{\left(  \lambda_{0},\lambda_{1}\right)
}^{\prime}\left(  -h\right)  }{\Phi_{\left(  \lambda_{0},\lambda_{1}\right)
}\left(  h\right)  }+\frac{1}{\Phi_{\left(  \lambda_{0},\lambda_{1}\right)
}\left(  h\right)  };
\]
we have used the fact that $\Phi_{\left(  \lambda_{0},\lambda_{1}\right)
}^{\prime}\left(  0\right)  =1$ and equation (\ref{eqsym}). We multiply the
last equation by $\Phi_{\left(  \lambda_{0},\lambda_{1}\right)  }\left(
h\right)  >0.$ Then it suffices to show that for all $h>0$
\[
g\left(  h\right)  :=1-e^{\left(  \lambda_{0}+\lambda_{1}\right)  h}%
\Phi_{\left(  \lambda_{0},\lambda_{1}\right)  }^{\prime}\left(  -h\right)
<0.
\]
Consider the case $\lambda_{1}\neq\lambda_{0},$ then
\[
g\left(  h\right)  =1-e^{\left(  \lambda_{0}+\lambda_{1}\right)  h}%
\frac{\lambda_{1}e^{\lambda_{1}\left(  -h\right)  }-\lambda_{0}e^{\lambda
_{0}\left(  -h\right)  }}{\lambda_{1}-\lambda_{0}}=1-\frac{\lambda
_{1}e^{\lambda_{0}h}-\lambda_{0}e^{\lambda_{1}h}}{\lambda_{1}-\lambda_{0}}.
\]
It follows that $g^{\prime}\left(  h\right)  =\lambda_{0}\lambda_{1}%
\Phi_{\left(  \lambda_{0},\lambda_{1}\right)  }\left(  h\right)  \leq0$ (here
we use that $\lambda_{0}\lambda_{1}\leq0)$ and we conclude that $g$ is
decreasing. Since $g\left(  0\right)  =0$ we conclude that $g\left(  h\right)
$ is negative for $h>0$, which ends the proof in the case $\lambda_{1}%
\neq\lambda_{0}$.

In the case $\lambda_{1}=\lambda_{0}$ we see that $\lambda_{1}=\lambda_{0}=0$,
which is the well known polynomial case.
\end{proof}

In the case of two positive frequencies one may obtain surprising effects if
we do not restrict the interval lengths of the partition $t_{1}<...<t_{n}.$ We
have seen that in this case the fundamental function $\Phi_{\left(
\lambda_{0},\lambda_{1}\right)  }$ may not be increasing and a plot for the
defining formula of a hat functions $H_{j}\left(  t\right)  $ looks like the
graph on Figure $5$.

%


\begin{figure}
[h] 
\begin{center}
\includegraphics[
height=2.1423in,
width=4.1945in
]%
{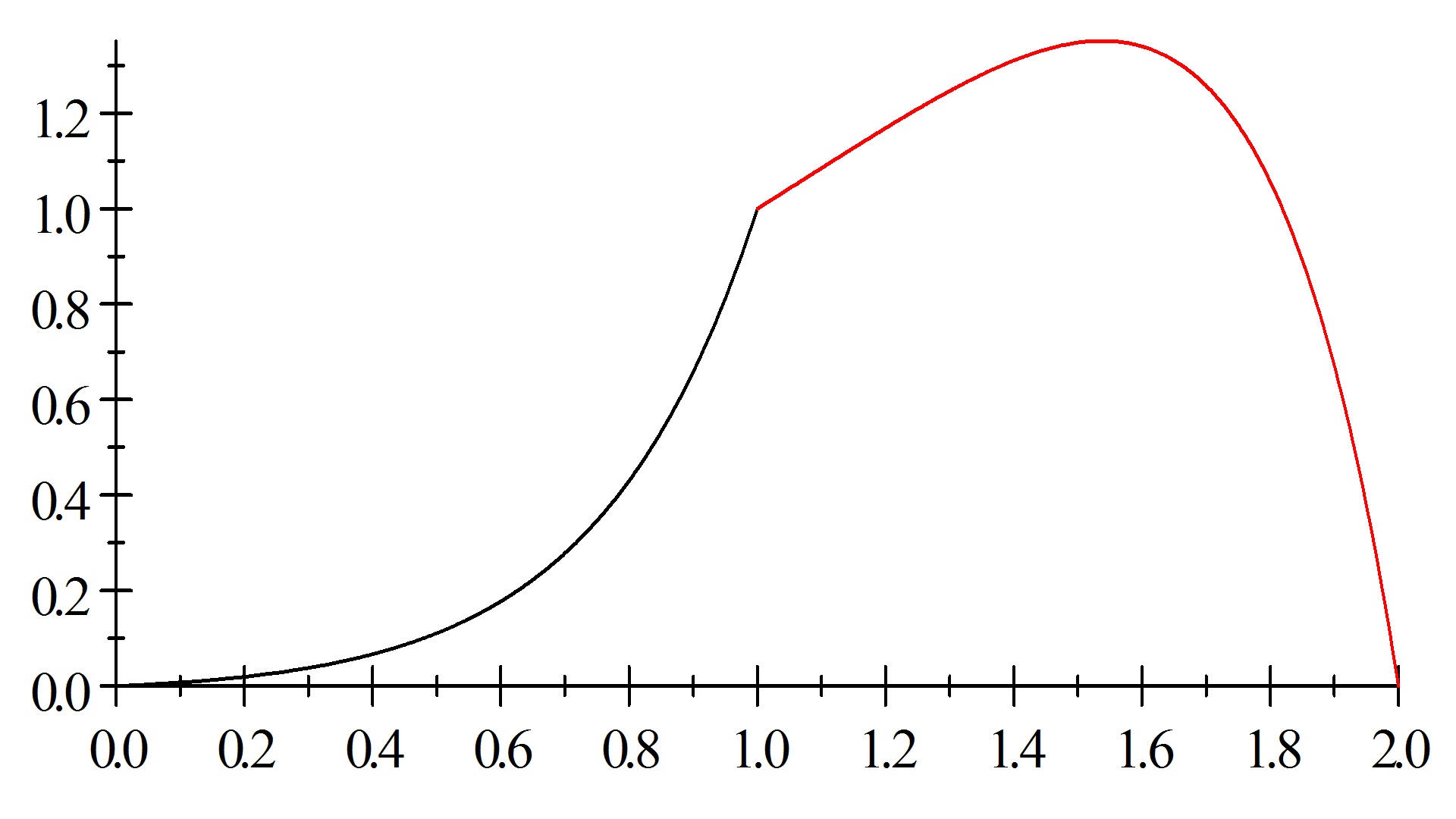}%
\caption{"Hat type function" for $\left(  \lambda_{0},\lambda
_{1}\right)  =\left(  4,1\right)  $}%
\end{center}
\end{figure}

However, if we restrict the size $\delta$ of the partition, the function
$\Phi_{\left(  \lambda_{0},\lambda_{1}\right)  }$ is increasing on $\left[
-\delta,\delta\right]  ,$ and this excludes such pathological behaviour of the
hat function.

\section{\label{S4}Error estimate for interpolation with piecewise exponential
splines of order $2$ in the uniform norm}

For the approximation of $f\in C^{2}\left[  t_{1},t_{n}\right]  $ by linear
splines the following estimate is well known:
\begin{equation}
\left\Vert f-I_{2}\left(  f\right)  \right\Vert _{\left[  t_{1},t_{n}\right]
}\leq\frac{1}{8}\max_{j=1,..n-1}\left\vert t_{j+1}-t_{j}\right\vert ^{2}%
\cdot\max_{\theta\in\left[  t_{1},t_{n}\right]  }\left\vert f^{\prime\prime
}\left(  \theta\right)  \right\vert , \label{eqClass1}%
\end{equation}
see e.g. \cite[p. 31]{deBoor}. The estimate (\ref{eqClass1}) is a simple
consequence from two facts: (i) $g\left(  t\right)  =f\left(  t\right)
-I_{2}\left(  f\right)  \left(  t\right)  $ vanishes at $t_{j}$ and $t_{j+1},$
and (ii) the following inequality is valid for all $g\in C^{2}\left[
a,b\right]  $ with $g\left(  a\right)  =g\left(  b\right)  =0$:
\begin{equation}
\left\vert g\left(  t\right)  \right\vert \leq\frac{1}{8}\left\vert
b-a\right\vert ^{2}\cdot\max_{\theta\in\left[  a,b\right]  }\left\vert
g^{\prime\prime}\left(  \theta\right)  \right\vert \text{.} \label{eqClass2}%
\end{equation}
We will find analogues of (\ref{eqClass2}) in the context of $L$-splines where
we will replace the second derivative in (\ref{eqClass2}) by
\[
\max_{\theta\in\left[  a,b\right]  }\left\vert L_{\left(  \lambda_{0}%
,\lambda_{1}\right)  }f\left(  \theta\right)  \right\vert \text{.}%
\]
We need now the following fact (see Remark \ref{RemarkOMEGA} in the Appendix):
if $\lambda_{0},\lambda_{1}$ are real and $a<b$ then there exists an
exponential polynomial $\Omega_{\lambda_{0},\lambda_{1},0}^{a,b} $ in
$E\left(  \lambda_{0},\lambda_{1},0\right)  $ such that
\begin{equation}
\Omega_{\lambda_{0},\lambda_{1},0}^{a,b}\left(  a\right)  =\Omega_{\lambda
_{0},\lambda_{1},0}^{a,b}\left(  b\right)  =0\text{ and }L_{\left(
\lambda_{0},\lambda_{1}\right)  }\Omega_{\lambda_{0},\lambda_{1},0}^{a,b}=-1.
\label{eqdefOmega}%
\end{equation}
We will further exploit the fact that $L_{\left(  \lambda_{0},\lambda
_{1}\right)  }\Omega_{\lambda_{0},\lambda_{1},0}^{a,b}$ is a constant
function. In the polynomial case (i.e. $\lambda_{0}=\lambda_{1}=0)$ we have
\[
\Omega_{0,0,0}^{a,b}\left(  t\right)  =\frac{1}{2}\left(  t-a\right)  \left(
b-t\right)  .
\]
If $\lambda_{0}\neq\lambda_{1}$ are both non-zero one can define
$\Omega_{\lambda_{0},\lambda_{1},0}^{a,b}\left(  t\right)  $ by putting
\begin{equation}
\Omega_{\lambda_{0},\lambda_{1},0}^{a,b}=\frac{\Phi_{\left(  \lambda
_{0},\lambda_{1}\right)  }\left(  t-a\right)  -e^{\left(  \lambda_{0}%
+\lambda_{1}\right)  \left(  b-a\right)  }\Phi_{\left(  \lambda_{0}%
,\lambda_{1}\right)  }\left(  t-b\right)  }{\lambda_{0}\lambda_{1}%
\Phi_{\left(  \lambda_{0},\lambda_{1}\right)  }\left(  b-a\right)  }-\frac
{1}{\lambda_{0}\lambda_{1}}. \label{eqFormulaOmega}%
\end{equation}

\begin{theorem}
The smallest constant $C$ such that for all $f\in C^{2}\left[  a,b\right]  $
with $f\left(  a\right)  =f\left(  b\right)  =0$ the following inequality
holds
\begin{equation}
\left\vert f\left(  t\right)  \right\vert \leq C\max_{\theta\in\left[
a,b\right]  }\left\vert L_{\left(  \lambda_{0},\lambda_{1}\right)  }f\left(
\theta\right)  \right\vert \text{ } \label{eqmaxderiv}%
\end{equation}
for all $t\in\left[  a,b\right]  $ is given by
\begin{equation}
C=M_{\lambda_{0},\lambda_{1}}^{a,b}:=\max_{t\in\left[  a,b\right]  }\left\vert
\Omega_{\lambda_{0},\lambda_{1},0}^{a,b}\left(  t\right)  \right\vert .
\label{eqDefLambda}%
\end{equation}
Moreover $\Omega_{\lambda_{0},\lambda_{1},0}^{a,b}$ is positive on $\left(
a,b\right)  .$
\end{theorem}

\begin{proof}
Suppose that $C$ is the best constant such that (\ref{eqmaxderiv}) holds for
all $f\in C^{2}\left[  a,b\right]  $ with $f\left(  a\right)  =f\left(
b\right)  =0.$ Inserting $f\left(  t\right)  =\Omega_{\lambda_{0},\lambda
_{1},0}^{a,b}$ gives
\[
\left\vert \Omega_{\lambda_{0},\lambda_{1},0}^{a,b}\left(  t\right)
\right\vert \leq C
\]
since $L_{\left(  \lambda_{0},\lambda_{1}\right)  }f=-1.$ It follows that
$M_{\lambda_{0},\lambda_{1}}^{a,b}\leq C.$ In order to show equality we recall
that the Green function $G_{\lambda_{0},\lambda_{1}}^{a,b}\left(
t,\xi\right)  $ of the differential operator $L_{\left(  \lambda_{0}%
,\lambda_{1}\right)  }$ for vanishing boundary conditions (see \cite[p.
65]{BiRo}) is defined by
\[
G_{\lambda_{0},\lambda_{1}}^{a,b}\left(  t,\xi\right)  =\left\{
\begin{array}
[c]{ccc}%
f\left(  t\right)  \frac{g\left(  \xi\right)  }{w\left(  \xi\right)  } &  &
\text{for }t\in\left[  a,\xi\right] \\
g\left(  t\right)  \frac{f\left(  \xi\right)  }{w\left(  \xi\right)  } &  &
\text{for }t\in\left[  \xi,b\right]
\end{array}
\right.
\]
where $f\left(  t\right)  =\Phi_{\left(  \lambda_{0},\lambda_{1}\right)
}\left(  t-a\right)  $ vanishes in $t=a$ and $g\left(  t\right)
=\Phi_{\left(  \lambda_{0},\lambda_{1}\right)  }\left(  t-b\right)  $ vanishes
in $t=b,$ and $w\left(  t\right)  =f\left(  t\right)  g^{\prime}\left(
t\right)  -f^{\prime}\left(  t\right)  g\left(  t\right)  $ is the Wronski
determinant. Note that $f\left(  t\right)  \geq0$ for all $t\in\left(
a,b\right)  $ and $g\left(  t\right)  \leq0$ for all $t\in\left(  a,b\right)
.$ A calculation shows that%
\[
w\left(  t\right)  =-e^{\left(  \lambda_{0}+\lambda_{1}\right)  \left(
t-a\right)  }\Phi_{\left(  \lambda_{0},\lambda_{1}\right)  }\left(
a-b\right)  \geq0\text{ for all }t\in\left(  a,b\right)  .
\]
Hence $\xi\longmapsto G_{\lambda_{0},\lambda_{1}}^{a,b}\left(  t,\xi\right)
\leq0$ for all $\xi\in\left(  a,b\right)  .$ It is known that for any function
$f\in C^{2}\left[  a,b\right]  $ with $f\left(  a\right)  =f\left(  b\right)
=0$ the representation
\begin{equation}
f\left(  t\right)  =\int_{a}^{b}G_{\lambda_{0},\lambda_{1}}^{a,b}\left(
t,\xi\right)  \cdot L_{\left(  \lambda_{0},\lambda_{1}\right)  }\left(
f\right)  \left(  \xi\right)  d\xi\label{eqrepresent}%
\end{equation}
holds. Let us take $f\left(  t\right)  =\Omega_{\lambda_{0},\lambda_{1}%
,0}^{a,b}\left(  t\right)  $ in (\ref{eqrepresent}). Since $L_{\left(
\lambda_{0},\lambda_{1}\right)  }\Omega_{\lambda_{0},\lambda_{1},0}%
^{a,b}\left(  t\right)  =-1$ we infer that%
\begin{equation}
\Omega_{\lambda_{0},\lambda_{1},0}^{a,b}\left(  t\right)  =-\int_{a}%
^{b}G_{\lambda_{0},\lambda_{1}}^{a,b}\left(  t,\xi\right)  d\xi.
\label{eqrepresentOmega}%
\end{equation}
It follows that $\Omega_{\lambda_{0},\lambda_{1},0}^{a,b}$ is positive on
$\left(  a,b\right)  $ since $\xi\longmapsto G_{\lambda_{0},\lambda_{1}}%
^{a,b}\left(  t,\xi\right)  \leq0$ for all $\xi\in\left(  a,b\right)  .$ From
(\ref{eqrepresent}) we see that
\begin{align*}
\left\vert f\left(  t\right)  \right\vert  &  \leq\int_{a}^{b}\left\vert
G_{\lambda_{0},\lambda_{1}}^{a,b}\left(  t,\xi\right)  \right\vert d\xi
\cdot\max_{\theta\in\left[  a,b\right]  }\left\vert L_{\left(  \lambda
_{0},\lambda_{1}\right)  }f\left(  \theta\right)  \right\vert \\
&  =\left\vert \Omega_{\lambda_{0},\lambda_{1},0}^{a,b}\left(  t\right)
\right\vert \cdot\max_{\theta\in\left[  a,b\right]  }\left\vert L_{\left(
\lambda_{0},\lambda_{1}\right)  }f\left(  \theta\right)  \right\vert .
\end{align*}
This implies that
\[
\left\Vert f\right\Vert _{\left[  a,b\right]  }\leq\max_{t\in\left[
a,b\right]  }\Omega_{\lambda_{0},\lambda_{1},0}^{a,b}\left(  t\right)
\cdot\max_{\theta\in\left[  a,b\right]  }\left\vert L_{\left(  \lambda
_{0},\lambda_{1}\right)  }f\left(  \theta\right)  \right\vert
\]
and therefore $C\leq M_{\lambda_{0},\lambda_{1}}^{a,b}.$ This ends the proof.
\end{proof}

In the case of symmetric frequencies we can determine the best constant:

\begin{theorem}
\label{ThmMOmega}For $\left(  \lambda_{0},\lambda_{1}\right)  =\left(
-\xi,\xi\right)  $ the following identity holds
\begin{equation}
M_{\xi,-\xi}^{a,b}=\max_{t\in\left[  a,b\right]  }\Omega_{\xi,-\xi,0}%
^{a,b}\left(  t\right)  =\left(  b-a\right)  ^{2}\cdot M^{\ast}\left(
\xi\left(  b-a\right)  \right)  \label{eqOmegaMax}%
\end{equation}
where
\begin{equation}
M^{\ast}\left(  x\right)  =\frac{\sinh x-2\sinh\left(  x/2\right)  }%
{x^{2}\sinh x}, \label{eqDefentireF}%
\end{equation}
and the following estimate holds:%
\begin{equation}
\Omega_{\xi,-\xi,0}^{a,b}\left(  t\right)  \leq\frac{1}{2}\left(  t-a\right)
\left(  b-t\right)  \leq\frac{1}{8}\left(  b-a\right)  ^{2}.
\label{eqOmegaest}%
\end{equation}

\end{theorem}

\begin{proof}
It is easy to see that the following function
\begin{equation}
f\left(  t\right)  :=\frac{1}{\xi^{2}}\left(  1-\frac{\sinh\xi\left(
t-a\right)  }{\sinh\xi\left(  b-a\right)  }+\frac{\sinh\xi\left(  t-b\right)
}{\sinh\xi\left(  b-a\right)  }\right)  . \label{eqDeff}%
\end{equation}
is equal to $\Omega_{\xi,-\xi,0}^{a,b}\left(  t\right)  $ since it is an
exponential polynomial in $E\left(  \xi,-\xi,0\right)  $ which vanishes in
$t=a $ and $t=b$ and $L_{\left(  \xi,-\xi\right)  }f\left(  t\right)  =-1.$
Further we see that
\[
F\left(  t\right)  :=\frac{1}{2}\left(  t-a\right)  \left(  b-t\right)
-f\left(  t\right)
\]
is a concave function since%
\[
F^{\prime\prime}\left(  t\right)  =-1+\frac{\sinh\xi\left(  t-a\right)
}{\sinh\xi\left(  b-a\right)  }-\frac{\sinh\xi\left(  t-b\right)  }{\sinh
\xi\left(  b-a\right)  }=-\xi^{2}\Omega_{\xi,-\xi}^{a,b}\left(  t\right)  <0.
\]
Since $F\left(  a\right)  =F\left(  b\right)  =0$ it follows that $F\left(
t\right)  \geq0$ for all $t\in\left[  a,b\right]  ,$ which proves
(\ref{eqOmegaest}). It is easy to see that $f$ has only one critical point,
namely $t_{\ast}=\left(  a+b\right)  /2$ which leads to the maximum of $f.$
Hence
\[
\max_{t\in\left[  a,b\right]  }f\left(  t\right)  =f\left(  t_{\ast}\right)
=\frac{1}{\xi^{2}}\left(  1-\frac{2\sinh\xi\left(  b-a\right)  /2}{\sinh
\xi\left(  b-a\right)  }\right)
\]
Now it is easy to derive formula (\ref{eqDefentireF}).
\end{proof}

The following picture shows that $\Omega_{\left(  1,-30,0\right)  }^{0,1},$
defined on the unit interval $\left[  0,1\right]  ,$ is not symmetric, while
the Green function on the diagonal, $G_{\left(  1,-30\right)  }^{0,1}$ is
symmetric and strictly larger. Both are smaller or equal than $t\left(
1-t\right)  $ (but not $\frac{1}{2}t\left(  1-t\right)  ,$ which is different
from the polynomial case), as seen from Figure $6$. 

%


\begin{figure}
[h] 
\begin{center}
\includegraphics[
height=2.1423in,
width=4.1945in
]%
{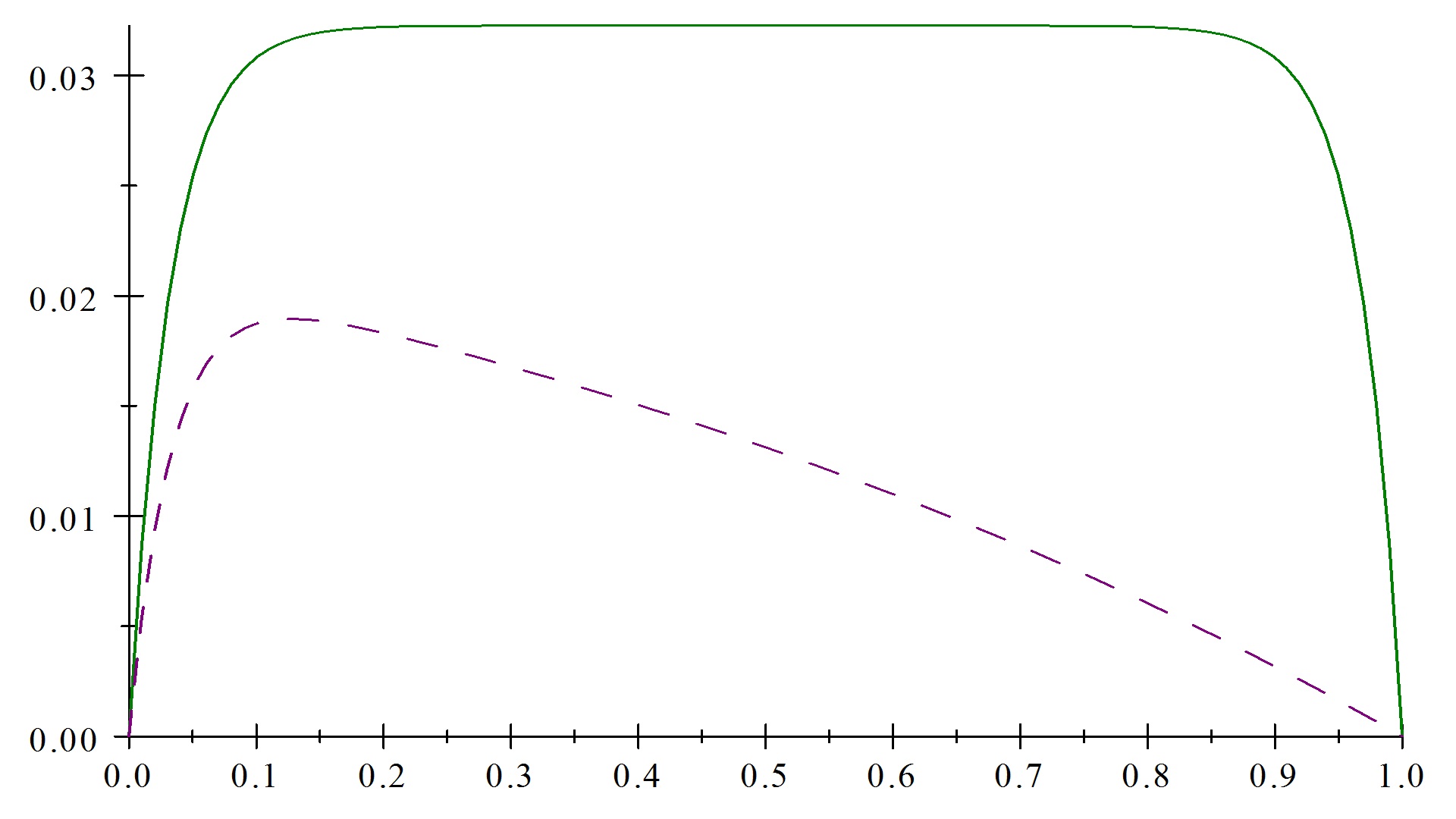}%
\caption{$\Omega_{\left(  1,-30,0\right)  }^{0,1}\left(  t\right)  $
is dash, and $G_{1,-30}^{0,1}\left(  t,t\right)  $ is a solid line.}%
\end{center}
\end{figure}

Next we give an upper bound for $\Omega_{\lambda_{0},\lambda_{1},0}%
^{a,b}\left(  t\right)  $:

\begin{proposition}
Let $\lambda_{0}\leq0\leq\lambda_{1}.$ Then for all $t\in\left[  a,b\right]
$
\[
\Omega_{\lambda_{0},\lambda_{1},0}^{a,b}\left(  t\right)  \leq\left(
b-a\right)  \cdot\left\vert G_{\lambda_{0},\lambda_{1}}^{a,b}\left(
t,t\right)  \right\vert
\]

\end{proposition}

\begin{proof}
Let $t\in\left(  a,b\right)  $ be fixed and let $a\leq\xi\leq t.$ Hence,
$t\in\left[  \xi,b\right]  $ and by the definition of the Green function we
obtain
\begin{align}
G_{\lambda_{0},\lambda_{1}}^{a,b}\left(  t,\xi\right)   &  =g\left(  t\right)
\frac{f\left(  \xi\right)  }{w\left(  \xi\right)  }=\frac{\Phi_{\left(
\lambda_{0},\lambda_{1}\right)  }\left(  t-b\right)  \Phi_{\left(  \lambda
_{0},\lambda_{1}\right)  }\left(  \xi-a\right)  }{-e^{\left(  \lambda
_{0}+\lambda_{1}\right)  \left(  \xi-a\right)  }\Phi_{\left(  \lambda
_{0},\lambda_{1}\right)  }\left(  a-b\right)  }\label{eqGreen1}\\
&  =\Phi_{\left(  \lambda_{0},\lambda_{1}\right)  }\left(  t-b\right)
\frac{\Phi_{\left(  -\lambda_{0},-\lambda_{1}\right)  }\left(  \xi-a\right)
}{\Phi_{\left(  -\lambda_{0},-\lambda_{1}\right)  }\left(  b-a\right)  }.
\label{eqGreen2}%
\end{align}
Since $\lambda_{0}\leq0\leq\lambda_{1},$ Propostion \ref{PropIncreasing} shows
that $\xi\longmapsto\Phi_{\left(  -\lambda_{0},-\lambda_{1}\right)  }\left(
\xi-a\right)  $ is increasing and positive on $\left(  a,t\right)  ,$ hence
for for $\xi\in\left[  a,t\right]  $ we obtain the inequality
\[
\left\vert G_{\lambda_{0},\lambda_{1}}^{a,b}\left(  t,\xi\right)  \right\vert
\leq\left\vert \Phi_{\left(  \lambda_{0},\lambda_{1}\right)  }\left(
t-b\right)  \right\vert \frac{\Phi_{\left(  -\lambda_{0},-\lambda_{1}\right)
}\left(  t-a\right)  }{\Phi_{\left(  -\lambda_{0},-\lambda_{1}\right)
}\left(  b-a\right)  }=\left\vert G_{\lambda_{0},\lambda_{1}}^{a,b}\left(
t,t\right)  \right\vert
\]
where for proving the last identity we have used the identities
(\ref{eqGreen1}) and (\ref{eqGreen2}) for $\xi=t.$

When $\xi\geq t$ then $t\in\left[  a,\xi\right]  $ and the definition of the
Green function shows that
\begin{align*}
G_{\lambda_{0},\lambda_{1}}^{a,b}\left(  t,\xi\right)   &  =f\left(  t\right)
\frac{g\left(  \xi\right)  }{w\left(  \xi\right)  }=\frac{\Phi_{\left(
\lambda_{0},\lambda_{1}\right)  }\left(  t-a\right)  \Phi_{\left(  \lambda
_{0},\lambda_{1}\right)  }\left(  \xi-b\right)  }{-e^{\left(  \lambda
_{0}+\lambda_{1}\right)  \left(  \xi-a\right)  }\Phi_{\left(  \lambda
_{0},\lambda_{1}\right)  }\left(  a-b\right)  }\\
&  =\frac{\Phi_{\left(  \lambda_{0},\lambda_{1}\right)  }\left(  t-a\right)
\Phi_{\left(  \lambda_{0},\lambda_{1}\right)  }\left(  \xi-b\right)
}{-e^{\left(  \lambda_{0}+\lambda_{1}\right)  \left(  \xi-b\right)
}e^{\left(  \lambda_{0}+\lambda_{1}\right)  \left(  b-a\right)  }\Phi_{\left(
\lambda_{0},\lambda_{1}\right)  }\left(  a-b\right)  }\\
&  =\frac{\Phi_{\left(  \lambda_{0},\lambda_{1}\right)  }\left(  t-a\right)
\Phi_{\left(  -\lambda_{0},-\lambda_{1}\right)  }\left(  \xi-b\right)
}{e^{\left(  \lambda_{0}+\lambda_{1}\right)  \left(  b-a\right)  }%
\Phi_{\left(  -\lambda_{0},-\lambda_{1}\right)  }\left(  b-a\right)  }.
\end{align*}
Since $\lambda_{0}\leq0\leq\lambda_{1}$ Propostion \ref{PropIncreasing} shows
that the function $x\longmapsto-\Phi_{\left(  -\lambda_{0},-\lambda
_{1}\right)  }\left(  x-b\right)  $ is decreasing for all $x,$ and for
$\xi\leq t$ we infer
\[
-G_{\lambda_{0},\lambda_{1}}^{a,b}\left(  t,\xi\right)  \leq\frac
{\Phi_{\left(  \lambda_{0},\lambda_{1}\right)  }\left(  t-a\right)  \left(
-\Phi_{\left(  -\lambda_{0},-\lambda_{1}\right)  }\left(  t-b\right)  \right)
}{e^{\left(  \lambda_{0}+\lambda_{1}\right)  \left(  b-a\right)  }%
\Phi_{\left(  -\lambda_{0},-\lambda_{1}\right)  }\left(  b-a\right)
}=-G_{\lambda_{0},\lambda_{1}}^{a,b}\left(  t,t\right)  .
\]
It follows that for fixed $t\in\left(  a,b\right)  $ and for all $\xi
\in\left[  a,b\right]  $ the inequality
\[
\left\vert G_{\lambda_{0},\lambda_{1}}^{a,b}\left(  t,\xi\right)  \right\vert
\leq\left\vert G_{\lambda_{0},\lambda_{1}}^{a,b}\left(  t,t\right)
\right\vert
\]
holds. Then the integral representation (\ref{eqrepresentOmega}) leads to the
estimate
\[
\Omega_{\lambda_{0},\lambda_{1},0}^{a,b}\left(  t\right)  \leq\left(
b-a\right)  \cdot\left\vert G_{\lambda_{0},\lambda_{1}}^{a,b}\left(
t,t\right)  \right\vert .
\]
The proof is complete.
\end{proof}

\begin{theorem}
Assume that $\lambda_{0}\leq\lambda_{1}$ are real numbers. Then
\[
\left(  b-a\right)  \cdot\left\vert G_{\left(  \lambda_{0},\lambda_{1}\right)
}^{a,b}\left(  t,t\right)  \right\vert \leq\frac{1}{4}\left(  b-a\right)
^{2}.
\]

\end{theorem}

\begin{proof}
The definition of the Green function in the last proof shows that
\[
G_{\left(  \lambda_{0},\lambda_{1}\right)  }^{a,b}\left(  t,t\right)
=\frac{\Phi_{\left(  \lambda_{0},\lambda_{1}\right)  }\left(  t-b\right)
\Phi_{\left(  -\lambda_{0},-\lambda_{1}\right)  }\left(  t-a\right)  }%
{\Phi_{\left(  -\lambda_{0},-\lambda_{1}\right)  }\left(  b-a\right)  }.
\]
It is clear that $F\left(  t\right)  :=\Phi_{\left(  \lambda_{0},\lambda
_{1}\right)  }\left(  t-b\right)  \Phi_{\left(  -\lambda_{0},-\lambda
_{1}\right)  }\left(  t-a\right)  $ is an exponential polynomial in $E\left(
0,\lambda_{1}-\lambda_{0},\lambda_{0}-\lambda_{1}\right)  .$ Further, $F$
vanishes at $t=a$ and $t=b$. For every constant $\lambda$ and function
$f\left(  t\right)  $ we define the differential operator $D_{\lambda
}f=df/dt-\lambda f.$ We apply the product rule $D_{\left(  \lambda+\mu\right)
}\left(  ab\right)  =D_{\lambda}a\cdot b+aD_{\mu}b$ and also formula
(\ref{rec0}) (in the Appendix) to obtain the equalities
\[
D_{\left(  -\lambda_{0}+\lambda_{1}\right)  }\left(  F\right)  =\Phi_{\left(
-\lambda_{1}\right)  }\left(  t-a\right)  \Phi_{\left(  \lambda_{0}%
,\lambda_{1}\right)  }\left(  t-b\right)  +\Phi_{\left(  -\lambda_{0}%
,-\lambda_{1}\right)  }\left(  t-a\right)  \Phi_{\left(  \lambda_{0}\right)
}\left(  t-b\right)
\]
and
\begin{align*}
&  D_{\left(  \lambda_{0}-\lambda_{1}\right)  }D_{\left(  -\lambda_{0}%
+\lambda_{1}\right)  }\left(  F\right) \\
&  =\Phi_{\left(  -\lambda_{1}\right)  }\left(  t-a\right)  \Phi_{\left(
\lambda_{1}\right)  }\left(  t-b\right)  +\Phi_{\left(  -\lambda_{0}\right)
}\left(  t-a\right)  \Phi_{\left(  \lambda_{0}\right)  }\left(  t-b\right) \\
&  =e^{-\lambda_{1}\left(  t-a\right)  }e^{\lambda_{1}\left(  t-b\right)
}+e^{-\lambda_{0}\left\langle t-a\right\rangle }e^{\lambda_{0}\left(
t-b\right)  }=e^{\lambda_{1}\left(  a-b\right)  }+e^{\lambda_{0}\left(
a-b\right)  }.
\end{align*}
Using the uniqueness property of the function $\Omega,$ we obtain the
following equality:
\[
G_{\left(  \lambda_{0},\lambda_{1}\right)  }^{a,b}\left(  t,t\right)
=\frac{e^{\lambda_{1}\left(  a-b\right)  }+e^{\lambda_{0}\left(  a-b\right)
}}{e^{-\lambda_{1}\left(  b-a\right)  }-e^{-\lambda_{1}\left(  b-a\right)  }%
}\left(  \lambda_{1}-\lambda_{0}\right)  \Omega_{\left(  \lambda_{1}%
-\lambda_{0},-\left(  \lambda_{1}-\lambda_{0}\right)  ,0\right)  }%
^{a,b}\left(  t\right)  .
\]
Further with $t=a-b,$ we obtain
\begin{align*}
\frac{e^{\lambda_{1}t}+e^{\lambda_{0}t}}{e^{\lambda_{1}t}-e^{\lambda_{0}t}}
&  =\frac{e^{\left(  -\lambda_{1}-\lambda_{0}\right)  t/2}}{e^{\left(
-\lambda_{1}-\lambda_{0}\right)  t/2}}\frac{e^{\lambda_{1}t}+e^{\lambda_{0}t}%
}{e^{\lambda_{1}t}-e^{\lambda_{0}t}}=\frac{e^{\left(  \lambda_{1}-\lambda
_{0}\right)  t/2}+e^{-\left(  \lambda_{1}-\lambda_{0}\right)  t/2}}{e^{\left(
\lambda_{1}-\lambda_{0}\right)  t/2}-e^{-\left(  \lambda_{1}-\lambda
_{0}\right)  t/2}}\\
&  =\frac{\cosh\left(  \left(  \lambda_{1}-\lambda_{0}\right)  t/2\right)
}{\sinh\left(  \left(  \lambda_{1}-\lambda_{0}\right)  t/2\right)  }.
\end{align*}
Since $\cosh\left(  x/2\right)  \sinh\left(  x/2\right)  =\frac{1}{2}\sinh x,$
it follows that
\[
\left(  b-a\right)  \cdot G_{\left(  \lambda_{0},\lambda_{1}\right)  }%
^{a,b}\left(  t,t\right)  =-T\left(  \left(  \lambda_{1}-\lambda_{0}\right)
\left(  b-a\right)  \right)  \cdot\Omega_{\left(  \lambda_{1}-\lambda
_{0},-\left(  \lambda_{1}-\lambda_{0}\right)  ,0\right)  }^{a,b}\left(
t\right)
\]
where we have put%
\[
T\left(  x\right)  =\frac{x\cosh x/2}{\sinh\left(  x/2\right)  }=\frac{1}%
{2}\frac{x\sinh x}{\sinh^{2}\left(  x/2\right)  }.
\]
Further, equation (\ref{eqOmegaMax}) and definition (\ref{eqDefentireF}) lead
to the estimate%
\[
\left(  b-a\right)  \cdot\left\vert G_{\left(  \lambda_{0},\lambda_{1}\right)
}^{a,b}\left(  t,t\right)  \right\vert \leq\left(  b-a\right)  ^{2}\cdot
T\left(  \left(  \lambda_{1}-\lambda_{0}\right)  \left(  b-a\right)  \right)
\cdot M^{\ast}\left(  \left(  \lambda_{1}-\lambda_{0}\right)  \left(
b-a\right)  \right)
\]
and%
\begin{align*}
T\left(  x\right)  M^{\ast}\left(  x\right)   &  =\frac{1}{2}\frac{x\sinh
x}{\sinh^{2}\left(  x/2\right)  }\frac{\sinh x-2\sinh\left(  x/2\right)
}{x^{2}\sinh x}\\
&  =\frac{1}{2}\frac{\sinh x-2\sinh\left(  x/2\right)  }{x\sinh^{2}\left(
x/2\right)  }\leq\frac{1}{4}.
\end{align*}

\end{proof}

Finally, we apply the above estimates to piecewise exponential splines in the
following result.

\begin{theorem}
Let $\lambda_{0,j}\leq\lambda_{1,j}$ be given for $j=1,...,n-1$, and define
$\varphi_{j}=\Phi_{\left(  \lambda_{0,j},\lambda_{1,j}\right)  }$. Then the
following estimate
\[
\left\Vert f-I_{2}\left(  f\right)  \right\Vert _{\left[  t_{1},t_{n}\right]
}\leq\max_{j=1,...n}M_{\lambda_{0,j},\lambda_{1,j}}^{t_{j},t_{j+1}}%
\max_{j=1,...,n-1}\max_{\theta\in\left[  t_{j},t_{j+1}\right]  }\left\vert
L_{\left(  \lambda_{0,j},\lambda_{1,j}\right)  }f\left(  \theta\right)
\right\vert
\]
holds for $f\in C^{2}\left[  t_{1},t_{n}\right]  $ where $I_{2}\left(
f\right)  $ is piecewise exponential spline of order $2$ interpolating $f.$
\end{theorem}

\begin{proof}
For each $t\in\left[  t_{j},t_{j+1}\right]  $ we estimate the function
$f\left(  t\right)  -I_{2}\left(  f\right)  \left(  t\right)  .$ Since this
function vanishes at $t_{j}$ and $t_{j+1}$ we can use (\ref{eqmaxderiv}) with
$C=M_{\lambda_{0},\lambda_{1}}^{a,b}$ defined in (\ref{eqDefLambda}).
\end{proof}

Let us put. $\Delta:=\max_{j=1,..n-1}\left\vert t_{j+1}-t_{j}\right\vert .$ We
have proved that $M_{\xi,-\xi}^{a,b}\leq\frac{1}{8}\left(  b-a\right)  ^{2}$
for any real $\xi.$ For the piecewise exponential spline $I_{2}\left(
f\right)  $ with frequencies $\left(  \xi_{j},-\xi_{j}\right)  $ interpolating
the function $f$ at $t_{1}<...<t_{n}$ we obtain then the estimate
\[
\left\Vert f-I_{2}\left(  f\right)  \right\Vert _{\left[  t_{1},t_{n}\right]
}\leq\frac{\Delta^{2}}{8}\max_{j=1,...,n-1}\max_{\theta\in\left[
t_{j},t_{j+1}\right]  }\left\vert L_{\left(  \xi_{j},-\xi_{j}\right)
}f\left(  \theta\right)  \right\vert .
\]
In the case of $\xi=0$ this reduces to the well-known classical estimate
(\ref{eqClass1}).

\section{\label{S5}Estimate for the best $L^{2}$-approximation}

We denote by $C\left(  X\right)  $ the space of all continuous complex-valued
functions on a compact space $X.$ The maximum norm is denoted by $\left\Vert
f\right\Vert _{X}=\max_{x\in X}\left\vert f\left(  x\right)  \right\vert ,$
and we define the weighted inner product
\begin{equation}
\left\langle f,g\right\rangle _{w}=\int_{X}f\left(  x\right)  \overline
{g\left(  x\right)  }w\left(  x\right)  dx \label{eqinner}%
\end{equation}
where $w\left(  x\right)  $ is a positive continuous function. We denote by
$\left\Vert f\right\Vert _{w}=\sqrt{\left\langle f,f\right\rangle _{w}}$ the
induced norm.

For any finite dimensional subspace $U_{n}$ of $C\left(  X\right)  $ and $f\in
C\left(  X\right)  $ we define by $P^{U_{n}}\left(  f\right)  $ the best
$L^{2}$-approximation from $U_{n},$ so
\[
\left\Vert f-P^{U_{n}}\left(  f\right)  \right\Vert _{w}=\inf\left\{
\left\Vert f-h\right\Vert _{w}:h\in U_{n}\right\}  .
\]
Then $P^{U_{n}}:C\left(  X\right)  \rightarrow C\left(  X\right)  $ defined by
$f\longmapsto P^{U_{n}}\left(  f\right)  $ is linear operator and a
projection. The operator norm of $P^{U_{n}}$ with respect to the uniform norm
is defined as
\[
\left\Vert P^{U_{n}}\right\Vert _{\text{op}}=\sup_{f\in C\left(  X\right)
}\frac{\left\Vert P^{U_{n}}\left(  f\right)  \right\Vert _{X}}{\left\Vert
f\right\Vert _{X}}\in\left[  0,\infty\right]  .
\]
The operator norm $\left\Vert P^{U_{n}}\right\Vert _{\text{op}}$ is a useful
tool in approximation theory since for all $f\in C\left(  X\right)  $
\begin{equation}
\text{dist}_{C\left(  X\right)  }\left(  f,U_{n}\right)  \leq\left\Vert
f-P^{U_{n}}f\right\Vert _{X}\leq\left(  1+\left\Vert P^{U_{n}}\right\Vert
\right)  \text{dist}_{C\left(  X\right)  }\left(  f,U_{n}\right)
\label{eqSecondSt}%
\end{equation}
where we have defined dist$_{C\left(  X\right)  }\left(  f,U_{n}\right)
=\inf_{g\in U_{n}}\left\Vert f-g\right\Vert _{X}.$ This inequality is
elementary: we can estimate for any $g\in U_{n}$, $f\in C\left(  X\right)  $
and $x\in X$
\begin{align*}
\left\vert f\left(  x\right)  -P^{U_{n}}\left(  f\right)  \left(  x\right)
\right\vert  &  \leq\left\vert f\left(  x\right)  -g\left(  x\right)
\right\vert +\left\vert P^{U_{n}}\left(  g-f\right)  \left(  x\right)
\right\vert \\
&  \leq\left\Vert f-g\right\Vert _{X}+\left\Vert P^{U_{n}}\right\Vert
\left\Vert f-g\right\Vert _{X}.
\end{align*}
using the fact that $g=P^{U_{n}}g$ for any $g\in U_{n}$.

\begin{definition}
Assume that $U_{n}$ is a subspace generated by linearly independent functions
$B_{1},...,B_{n}$ in $C\left(  X\right)  $. Then we define a matrix $S=\left(
s_{i,j}\right)  _{i,j=1,...,n}$ by
\[
s_{i,j}=\left\langle B_{i},B_{j}\right\rangle _{w}=\int_{X}B_{i}\left(
x\right)  \overline{B_{j}\left(  x\right)  }w\left(  x\right)  dx.
\]
We say that $S$ is tridiagonal, if $s_{i,j}$ is zero whenever $j\neq
i,i+1,i-1,$ and we say that $S$ is row diagonally dominant with dominance
factor $c\in\left(  0,1\right)  $ if
\[
\left\vert s_{j-1,j}\right\vert +\left\vert s_{j,j+1}\right\vert \leq
c\left\vert s_{j,j}\right\vert \text{ for }j=1,...n\text{ }%
\]
with the convention that $s_{0,1}=0$ and $s_{n,n+1}=0.$ Alternatively, we
require $\left\vert s_{1,2}\right\vert \leq cs_{11}$ and $\left\vert
s_{n-1,n}\right\vert \leq cs_{n,n}$.
\end{definition}

The following result is a generalization of well known techniques used in the
error estimates of cubic splines, see e.g. \cite[p. 34]{deBoor}. We include
the proof for convenience for the reader but emphasize that it follows closely
the lines in \cite[p. 34]{deBoor}.

\begin{theorem}
\label{ThmTridiagonal}Let $U_{n}$ be the subspace generated by linearly
independent functions $B_{1},...,B_{n}$ in $C\left(  X\right)  $ and assume
that the matrix $S=\left(  s_{i,j}\right)  _{i,j=1,...,n}$ is tridiagonal and
row diagonally dominant with dominance constant $c\in\left(  0,1\right)  .$
Then
\begin{equation}
\left\Vert P^{U_{n}}\right\Vert _{\text{op}}\leq\frac{\max_{x\in X}\sum
_{j=1}^{n}\left\vert B_{j}\left(  x\right)  \right\vert }{1-c}\max
_{j=1,..,n}\frac{\int_{X}\left\vert B_{j}\left(  y\right)  \right\vert
w\left(  y\right)  dy}{\int_{X}\left\vert B_{j}\left(  y\right)  \right\vert
^{2}w\left(  y\right)  dy}. \label{eqoperatornormFirst}%
\end{equation}

\end{theorem}

\begin{proof}
For $g\in C\left(  X\right)  $ and the best approximant $g^{\ast}:=P^{U_{n}%
}\left(  g\right)  $ to $U_{n}$ it is well known that $g^{\ast}-g$ is
orthogonal to $U_{n}$, so
\begin{equation}
\left\langle f,g^{\ast}-g\right\rangle _{w}=0,\text{ hence }\left\langle
f,g^{\ast}\right\rangle _{w}=\left\langle f,g\right\rangle _{w}
\label{eqLeastSquares}%
\end{equation}
for all $f\in U_{n}.$ We take $f=B_{i}$ in (\ref{eqLeastSquares}) and write
$g^{\ast}=\sum_{j=1}^{n}\alpha_{j}B_{j}.$ From (\ref{eqLeastSquares}) it
follows that the coefficients $\alpha_{1},...,\alpha_{n}$ satisfy the
equations
\[
\sum_{j=1}^{n}\alpha_{j}s_{i,j}=\left\langle B_{i},g\right\rangle _{w}%
=:\beta_{i}%
\]
for $i=1,...,n.$ Since the basis $B_{1},....,B_{n}$ is tridiagonal we arrive
after division by $s_{ii}>0$ at
\begin{equation}
\alpha_{i-1}\frac{s_{i-1,i}}{s_{ii}}+\alpha_{i}+\alpha_{i+1}\frac{s_{i,i+1}%
}{s_{ii}}=\frac{\beta_{i}}{s_{ii}}. \label{eqalpha4}%
\end{equation}
Let $j$ be the index such that $\alpha_{j}=\max_{i=1,...,n}\left\vert
\alpha_{i}\right\vert .$ Then equation (\ref{eqalpha4}) leads to
\[
\left\vert \alpha_{j}\right\vert \leq\left\vert \frac{\beta_{j}}{s_{j,j}%
}\right\vert +\left\vert \alpha_{j}\right\vert \frac{\left\vert s_{j-1,j}%
\right\vert +\left\vert s_{j,j+1}\right\vert }{s_{j,j}}\leq\left\vert
\frac{\beta_{j}}{s_{j,j}}\right\vert +\left\vert \alpha_{j}\right\vert c.
\]
Thus, $\left(  1-c\right)  \left\vert \alpha_{j}\right\vert \leq\left\vert
\frac{\beta_{j}}{s_{j,j}}\right\vert $ and we infer that
\[
\max_{i=1,...,n}\left\vert \alpha_{i}\right\vert =\left\vert \alpha
_{j}\right\vert \leq\frac{1}{1-c}\left\vert \frac{\beta_{j}}{s_{j,j}%
}\right\vert \leq\frac{1}{1-c}\max_{i=1,...,n}\left\vert \frac{\beta_{i}%
}{s_{i,i}}\right\vert .
\]
It follows that
\[
\left\vert P^{U_{n}}\left(  g\right)  \left(  x\right)  \right\vert
=\left\vert g^{\ast}\left(  x\right)  \right\vert \leq\sum_{j=1}^{n}\left\vert
\alpha_{j}\right\vert \left\vert B_{j}\left(  x\right)  \right\vert \leq
\frac{\sum_{j=1}^{n}\left\vert B_{j}\left(  x\right)  \right\vert }{1-c}%
\max_{i=1,...,n}\left\vert \frac{\beta_{i}}{s_{i,i}}\right\vert .
\]
The statement of the theorem follows now from the estimate
\[
\left\vert \frac{\beta_{j}}{s_{j,j}}\right\vert =\frac{\left\vert \left\langle
B_{j},g\right\rangle _{w}\right\vert }{\left\langle B_{j},B_{j}\right\rangle
_{w}}=\frac{\left\vert \int_{X}B_{j}\left(  y\right)  g\left(  y\right)
w\left(  y\right)  dy\right\vert }{\int_{X}\left\vert B_{j}\left(  y\right)
\right\vert ^{2}w\left(  y\right)  dy}\leq\left\Vert g\right\Vert _{X}%
\frac{\int_{X}\left\vert B_{j}\left(  y\right)  \right\vert w\left(  y\right)
dy}{\int_{X}\left\vert B_{j}\left(  y\right)  \right\vert ^{2}w\left(
y\right)  dy}.
\]

\end{proof}

Assume now that $\varphi_{j}:\left[  -\delta,\delta\right]  \rightarrow
\mathbb{C}$ are continuous strictly increasing functions with $\varphi
_{j}\left(  0\right)  =0$ for $j=1,...,n-1,$ and define the hat functions
$H_{1},...,H_{n}$ for $t_{1}<...<t_{n}$ as in Section 2. Let $U_{n}$ be the
linear span of $H_{1},...,H_{n}.$ We consider the inner product
\begin{equation}
\left\langle f,g\right\rangle _{p}=\int_{t_{1}}^{t_{n}}f\left(  t\right)
g\left(  t\right)  e^{pt}dt. \label{eqinnerprod}%
\end{equation}
We introduce the following notations:
\begin{align}
A_{\varphi}\left(  h\right)   &  =\frac{\varphi\left(  h\right)  }%
{\varphi\left(  -h\right)  }\frac{\int_{0}^{h}\varphi\left(  \tau-h\right)
\varphi\left(  \tau\right)  e^{p\tau}dt}{\int_{0}^{h}\left(  \varphi\left(
\tau\right)  \right)  ^{2}e^{p\tau}dt},\label{eqA}\\
B_{\varphi}\left(  h\right)   &  =\frac{\varphi\left(  -h\right)  }%
{\varphi\left(  h\right)  }\frac{\int_{0}^{h}\varphi\left(  \tau-h\right)
\varphi\left(  \tau\right)  e^{p\tau}dt}{\int_{0}^{h}\left(  \varphi\left(
\tau-h\right)  \right)  ^{2}e^{p\tau}dt},\label{eqB}\\
C_{\varphi}\left(  h\right)   &  =\varphi\left(  h\right)  \frac{\int_{0}%
^{h}\varphi\left(  t\right)  e^{pt}dt}{\int_{t_{0}}^{h}\varphi\left(
t\right)  ^{2}e^{pt}dt},\label{eqC}\\
D_{\varphi}\left(  h\right)   &  =\varphi\left(  -h\right)  \frac{\int_{0}%
^{h}\varphi\left(  t-h\right)  e^{pt}dt}{\int_{0}^{h}\varphi\left(
t-h\right)  ^{2}e^{pt}dt}. \label{eqD}%
\end{align}

\begin{proposition}
\label{PropABCD}Assume that $\varphi_{j}:\left[  -\delta,\delta\right]
\rightarrow\mathbb{C}$ are continuous strictly increasing functions with
$\varphi_{j}\left(  0\right)  =0$ for $j=1,...,n-1,$ and let $t_{1}<...<t_{n}$
with $t_{j+1}-t_{j}\leq\delta$ for $j=1,...,n-1.$ Then the matrix $S=\left(
s_{i,j}\right)  _{i,j=1,...,n}$ with $s_{i,j}=\left\langle H_{i}%
,H_{j}\right\rangle _{p}$ is row diagonally dominant with constant
$c\in\left(  0,1\right)  $ if
\begin{equation}
\max_{j=1,...n-1}\left\{  A_{\varphi_{j}}\left(  h_{j}\right)  ,B_{\varphi
_{j}}\left(  h_{j}\right)  \right\}  \leq c \label{eqABc}%
\end{equation}
where $h_{j}:=t_{j+1}-t_{j}.$ Further, the following estimate holds
\begin{equation}
\left\Vert P^{U_{n}}\right\Vert _{\text{op}}\leq\max_{t\in\left[  t_{1}%
,t_{n}\right]  }\sum_{j=1}^{n}\left\vert H_{j}\left(  t\right)  \right\vert
\frac{\max_{j=1,...n-1}\left\{  C_{\varphi_{j}}\left(  h_{j}\right)
,D_{\varphi_{j}}\left(  h_{j}\right)  \right\}  }{1-c} \label{eqoperatornorm}%
\end{equation}

\end{proposition}

\begin{proof}
We shall apply Theorem \ref{ThmTridiagonal}. Let us define
\begin{align*}
c_{j}  &  =\int_{t_{j-1}}^{t_{j}}\left\vert \frac{\varphi_{j-1}\left(
t-t_{j-1}\right)  }{\varphi_{j-1}\left(  t_{j}-t_{j-1}\right)  }\right\vert
^{2}e^{pt}dt\text{ for }j=2,...,n,\\
d_{j}  &  =\int_{t_{j}}^{t_{j+1}}\left\vert \frac{\varphi_{j}\left(
t-t_{j+1}\right)  }{\varphi_{j}\left(  t_{j}-t_{j+1}\right)  }\right\vert
^{2}e^{pt}dt\text{ for }j=1,....n-1.
\end{align*}
Since $H_{j}$ has support in $\left[  t_{j-1},t_{j+1}\right]  $ for
$j=2,...,n-1$ we obtain
\begin{equation}
s_{j,j}=\int_{t_{1}}^{t_{n}}\left\vert H_{j}\left(  t\right)  \right\vert
^{2}e^{pt}dt=c_{j}+d_{j} \label{eqL2integral}%
\end{equation}
for $j=2,...,n-1,$ and for $j=1$ and $j=n$ we have
\[
s_{1,1}=\int_{t_{1}}^{t_{2}}\left\vert H_{1}\left(  t\right)  \right\vert
^{2}e^{pt}dt=d_{1}\text{ and }s_{n,n}=\int_{t_{n-1}}^{t_{n}}\left\vert
H_{n}\left(  t\right)  \right\vert ^{2}e^{pt}dt=c_{n}.
\]
Since the product function $H_{j-1}H_{j}$ has support in $\left[
t_{j-1},t_{j}\right]  $ we see that for $j=2,...,n$
\[
s_{j-1,j}=\int_{t_{j-1}}^{t_{j}}\frac{\varphi_{j-1}\left(  t-t_{j}\right)
}{\varphi_{j-1}\left(  t_{j-1}-t_{j}\right)  }\frac{\varphi_{j-1}\left(
t-t_{j-1}\right)  }{\varphi_{j-1}\left(  t_{j}-t_{j-1}\right)  }e^{pt}dt.
\]
Similarly, $H_{j}H_{j+1}$ has support in $\left[  t_{j},t_{j+1}\right]  $, and
for $j=1,...,n-1$
\[
s_{j,j+1}=\int_{t_{j}}^{t_{j+1}}\frac{\varphi_{j}\left(  t-t_{j+1}\right)
}{\varphi_{j}\left(  t_{j}-t_{j+1}\right)  }\frac{\varphi_{j}\left(
t-t_{j}\right)  }{\varphi_{j}\left(  t_{j+1}-t_{j}\right)  }e^{pt}dt.
\]
It follows that for $j=2,...,n-1$
\begin{equation}
\frac{\left\vert s_{j-1,j}\right\vert +\left\vert s_{j,j+1}\right\vert
}{s_{j,j}}=\frac{\left\vert s_{j-1,j}\right\vert +\left\vert s_{j,j+1}%
\right\vert }{c_{j}+d_{j}}\leq\max\left\{  \frac{\left\vert s_{j-1,j}%
\right\vert }{c_{j}},\frac{\left\vert s_{j,j+1}\right\vert }{d_{j}}\right\}  .
\label{eqsdominant}%
\end{equation}
Hence, in order to apply Theorem \ref{ThmTridiagonal}, it suffices to require
that
\[
\max\left\{  \frac{\left\vert s_{j,j+1}\right\vert }{c_{j+1}},\frac{\left\vert
s_{j,j+1}\right\vert }{d_{j}}\right\}  \leq c\text{ for all }%
j=1,...,n-1\text{.}%
\]
The last condition is satisfied since (\ref{eqABc}) holds and for
$j=1,...,n-1$
\[
\frac{s_{j,j+1}}{c_{j+1}}=\frac{\varphi_{j}\left(  t_{j+1}-t_{j}\right)
}{\varphi_{j}\left(  t_{j}-t_{j+1}\right)  }\frac{\int_{t_{j}}^{t_{j+1}%
}\varphi_{j}\left(  t-t_{j+1}\right)  \varphi_{j}\left(  t-t_{j}\right)
e^{pt}dt}{\int_{t_{j}}^{t_{j+1}}\left(  \varphi_{j}\left(  t-t_{j}\right)
\right)  ^{2}e^{pt}dt}=A_{\varphi_{j}}\left(  t_{j+1}-t_{j}\right)  .
\]
Similarly, for $j=1,...,n-1$, we have $\frac{s_{j,j+1}}{d_{j}}=B_{\varphi_{j}%
}\left(  t_{j+1}-t_{j}\right)  .$

The inequality (\ref{eqoperatornorm}) is a simple consequence of
(\ref{eqoperatornormFirst}) and the following claim:
\begin{equation}
\frac{\int_{t_{1}}^{t_{n}}H_{j}\left(  t\right)  w\left(  t\right)  dt}%
{\int_{t_{1}}^{t_{n}}H_{j}\left(  t\right)  ^{2}w\left(  t\right)  dt}\leq
\max_{j=1,...n-1}\left\{  C_{\varphi_{j}}\left(  t_{j+1}-t_{j}\right)
,D_{\varphi_{j}}\left(  t_{j+1}-t_{j}\right)  \right\}  \label{eqBoundM}%
\end{equation}
for all $j=1,...n-1.$ Indeed, for $j=2,...,n-1$ we have
\[
\frac{\int_{t_{1}}^{t_{n}}H_{j}\left(  t\right)  e^{pt}dt}{\int_{t_{1}}%
^{t_{n}}H_{j}\left(  t\right)  ^{2}e^{pt}dt}=\frac{\sigma_{j-1,j}%
+\sigma_{j,j+1}}{c_{j}+d_{j}}\leq\max\left\{  \frac{\sigma_{j-1,j}}{c_{j}%
},\frac{\sigma_{j,j+1}}{d_{j}}\right\}
\]
where we define for $j=2,...,n-1$
\[
\sigma_{j-1,j}=\int_{t_{j-1}}^{t_{j}}\frac{\varphi_{j-1}\left(  t-t_{j-1}%
\right)  }{\varphi_{j-1}\left(  t_{j}-t_{j-1}\right)  }e^{pt}dt\text{ and
}\sigma_{j,j+1}=\int_{t_{j}}^{t_{j+1}}\frac{\varphi_{j}\left(  t-t_{j+1}%
\right)  }{\varphi_{j}\left(  t_{j}-t_{j+1}\right)  }e^{pt}dt.
\]
Now (\ref{eqBoundM}) follows from the fact that
\[
\frac{\sigma_{j,j+1}}{c_{j+j}}=\varphi_{j}\left(  t_{j+1}-t_{j}\right)
\frac{\int_{t_{j}}^{t_{j+1}}\varphi_{j}\left(  t-t_{j}\right)  e^{pt}dt}%
{\int_{t_{j}}^{t_{j+1}}\varphi_{j}\left(  t-t_{j}\right)  ^{2}e^{pt}%
dt}=C_{\varphi_{j}}\left(  t_{j+1}-t_{j}\right)  ,
\]
and $\frac{\sigma_{j,j+1}}{d_{j}}=D_{\varphi_{j}}\left(  t_{j+1}-t_{j}\right)
$ for $j=1,...,n-1$.
\end{proof}

\section{\label{S6}Estimate for the best $L^{2}$-approximation by piecewise
exponential splines of order $2$}

Let us recall the basic assumptions (i) and (ii) from Section 3, and add a new feature:

\begin{enumerate}
\item[(i)] Let $\delta>0.$ We assume that for the real numbers $\lambda
_{0,j}\leq\lambda_{1,j}$ the functions $\varphi_{j}:=\Phi_{\left(
\lambda_{0,j},\lambda_{1,j}\right)  }$, $j=1,...,n-1,$ are increasing on
$\left[  -\delta,\delta\right]  .$

\item[(ii)] For given $t_{1}<...<t_{n}$ with $\left\vert t_{j+1}%
-t_{j}\right\vert \leq\delta,$ for $j=1,..,n-1$ the corresponding hat
functions are denoted by $H_{1},....,H_{n}$, and their linear span is denoted
by $\mathcal{H}_{n}$.

\item[(iii)] For a real number $p$ we define on $C\left[  t_{1},t_{n}\right]
$ the inner product%
\[
\left\langle f,g\right\rangle _{p}:=\int_{t_{1}}^{t_{n}}f\left(  t\right)
g\left(  t\right)  e^{pt}dt.
\]

\end{enumerate}

The motivation for the introduction of a weight $e^{pt}$ in (iii) originates
from consideration of polysplines on annuli where the weight depends on the
dimension of the euclidean space.

Define the functions $T_{\left(  \lambda_{0},\lambda_{1}\right)  }^{\left(
p\right)  }:\mathbb{R}\rightarrow\mathbb{R}$ and $S_{\left(  \lambda
_{0},\lambda_{1}\right)  }^{\left(  p\right)  }:\mathbb{R}\rightarrow
\mathbb{R}$ by putting
\begin{align}
T_{\left(  \lambda_{0},\lambda_{1}\right)  }^{\left(  p\right)  }  &
=\frac{1}{2}\frac{\Phi_{\left(  \lambda_{0},\lambda_{1},-p-\lambda
_{0},-p-\lambda_{1}\right)  }}{\Phi_{\left(  \lambda_{0}-\lambda_{1}%
,\lambda_{1}-\lambda_{0},0,-p-\lambda_{0}-\lambda_{1}\right)  }},\label{eqT}\\
S_{\left(  \lambda_{0},\lambda_{1}\right)  }^{\left(  p\right)  }  &
=\frac{1}{2}\frac{\Phi_{\left(  -\lambda_{0},-\lambda_{1}\right)  }%
\Phi_{\left(  \lambda_{0},\lambda_{1},-p\right)  }}{\Phi_{\left(  \lambda
_{0}-\lambda_{1},\lambda_{1}-\lambda_{0},0,-p-\lambda_{0}-\lambda_{1}\right)
}}. \label{eqS}%
\end{align}
Here we recall that the function $\Phi_{\left(  \lambda_{0},\lambda
_{1},...,\lambda_{n}\right)  }$ denotes the fundamental function, see formula
(\ref{defPhi}) in the Appendix.

The following Theorem \ref{ThmMain1} is the main result of the present section
which will be proved as a consequence of Proposition \ref{PropABCD}. After
that we will prove a simple explicit estimate of the operator norm $\left\Vert
P^{\mathcal{H}_{n}}\right\Vert _{\text{op}}$ for sufficiently small grid sizes
$\max\left\vert t_{j+1}-t_{j}\right\vert .$ In the second part of the section
we will give estimates of the functions defined in (\ref{eqT}) and (\ref{eqS})
for different configurations of the constants $\left\{  \lambda_{0,j}%
,\lambda_{1,j}\right\}  $ and give an explicit form of the estimate of Theorem
\ref{ThmMain1}.

\begin{theorem}
\label{ThmMain1}In addition to (i)--(iii) assume that
\begin{equation}
c\left(  \delta\right)  :=\sup_{\left\vert h\right\vert \leq\delta}%
\max_{j=1,...,n-1}\left\vert T_{\left(  \lambda_{0,j},\lambda_{1,j}\right)
}^{\left(  p\right)  }\left(  h\right)  \right\vert <1. \label{eqFormulac}%
\end{equation}
Then for any partition $t_{1}<...<t_{n}$ with $t_{j+1}-t_{j}\leq\delta$ the
matrix $\left(  s_{i,j}\right)  _{i,j=1,...,n}$ defined by $s_{i,j}%
=\left\langle H_{i},H_{j}\right\rangle _{p}$ is row diagonally dominant and
\begin{equation}
\left\Vert P^{\mathcal{H}_{n}}\right\Vert _{\text{op}}\leq\max_{t\in\left[
t_{1},t_{n}\right]  }\sum_{j=1}^{n}\left\vert H_{j}\left(  t\right)
\right\vert \frac{\sup_{\left\vert h\right\vert \leq\delta}\max_{j=1,...,n-1}%
\left\vert S_{\left(  \lambda_{0,j},\lambda_{1,j}\right)  }^{\left(  p\right)
}\left(  h\right)  \right\vert }{1-c\left(  \delta\right)  }.
\label{eqoperatornorm2}%
\end{equation}

\end{theorem}

\begin{proof}
In view of Proposition \ref{PropABCD} it suffices to show that $A_{\varphi
}\left(  h\right)  =T_{\left(  \lambda_{0},\lambda_{1}\right)  }^{\left(
p\right)  }\left(  h\right)  $ and $B_{\varphi}\left(  h\right)  =T_{\left(
\lambda_{0},\lambda_{1}\right)  }^{\left(  p\right)  }\left(  -h\right)  $ for
any $h>0,$ and $C_{\varphi}\left(  h\right)  =S_{\left(  \lambda_{0}%
,\lambda_{1}\right)  }^{\left(  p\right)  }\left(  h\right)  $ and
$D_{\varphi}\left(  h\right)  =S_{\left(  \lambda_{0},\lambda_{1}\right)
}^{\left(  p\right)  }\left(  -h\right)  $ for any $h>0$ where $A_{\varphi}$,
$B_{\varphi}$, $C_{\varphi}$ and $D_{\varphi}$ are defined in (\ref{eqA}) to
(\ref{eqD}) for the function $\varphi=\Phi_{\left(  \lambda_{0},\lambda
_{1}\right)  }$.

In the sequel we often apply the following simple rules:
\begin{align}
e^{pt}\Phi_{\left(  \lambda_{0},....,\lambda_{N}\right)  }\left(  t\right)
&  =\Phi_{\left(  p+\lambda_{0},....,p+\lambda_{N}\right)  }\left(  t\right)
\label{eqTrans}\\
\Phi_{\left(  \lambda_{0},...,\lambda_{N}\right)  }\left(  -t\right)   &
=\left(  -1\right)  ^{N}\Phi_{\left(  -\lambda_{0},...,-\lambda_{N}\right)
}\left(  t\right) \label{eqMinus}\\
\Phi_{\left(  \lambda_{0},...,\lambda_{N}\right)  }\left(  pt\right)   &
=p^{N}\Phi_{\left(  p\lambda_{0},...,p\lambda_{N}\right)  }\left(  t\right)
\label{eqmultiple}%
\end{align}
Then (\ref{eqMinus}) and (\ref{eqTrans}) shows that
\[
\int_{0}^{h}\Phi_{\left(  \lambda_{0},\lambda_{1}\right)  }\left(  t-h\right)
\Phi_{\left(  \lambda_{0},\lambda_{1}\right)  }\left(  t\right)
e^{pt}dt=-\int_{0}^{h}\Phi_{\left(  -\lambda_{0},-\lambda_{1}\right)  }\left(
h-t\right)  \Phi_{\left(  p+\lambda_{0},p+\lambda_{1}\right)  }\left(
t\right)  dt.
\]
Theorem \ref{ThmApp1} in the appendix (applied to the right hand side) shows
that%
\begin{equation}
\int_{0}^{h}\Phi_{\left(  \lambda_{0},\lambda_{1}\right)  }\left(  t-h\right)
\Phi_{\left(  \lambda_{0},\lambda_{1}\right)  }\left(  t\right)
e^{pt}dt=-\Phi_{\left(  p+\lambda_{0},p+\lambda_{1},-\lambda_{0},-\lambda
_{1}\right)  }\left(  h\right)  . \label{eqApp1}%
\end{equation}
Further we have seen in (\ref{eqsym}) that
\begin{equation}
\frac{\Phi_{\left(  \lambda_{0},\lambda_{1}\right)  }\left(  h\right)  }%
{\Phi_{\left(  \lambda_{0},\lambda_{1}\right)  }\left(  -h\right)  }=-\frac
{1}{e^{-\left(  \lambda_{0}+\lambda_{1}\right)  h}}. \label{eqsymNew}%
\end{equation}
Now (\ref{eqApp1}), and the identity (\ref{eqId1}) in the appendix show that
\[
A_{\varphi}\left(  h\right)  =\frac{\varphi\left(  h\right)  }{\varphi\left(
-h\right)  }\frac{\int_{0}^{h}\varphi\left(  \tau-h\right)  \varphi\left(
\tau\right)  e^{p\tau}d\tau}{\int_{0}^{h}\left(  \varphi\left(  \tau\right)
\right)  ^{2}e^{p\tau}d\tau}=\frac{e^{-ph}\Phi_{\left(  p+\lambda
_{0},p+\lambda_{1},-\lambda_{0},-\lambda_{1}\right)  }\left(  h\right)
}{2e^{-\left(  \lambda_{0}+\lambda_{1}\right)  h}\Phi_{\left(  2\lambda
_{0},2\lambda_{1},\lambda_{0}+\lambda_{1},-p\right)  }\left(  h\right)  }.
\]
Apply the rule (\ref{eqTrans}) for the nominator and the denominator, then
\[
A_{\varphi}\left(  h\right)  =\frac{\Phi_{\left(  \lambda_{0},\lambda
_{1},-p-\lambda_{0},-p-\lambda_{1}\right)  }\left(  h\right)  }{2\Phi_{\left(
\lambda_{0}-\lambda_{1},\lambda_{1}-\lambda_{0},0,-p-\lambda_{0}-\lambda
_{1}\right)  }\left(  h\right)  }=T_{\left(  \lambda_{0},\lambda_{1}\right)
}^{\left(  p\right)  }\left(  h\right)  .
\]
Similarly, the definition of $B_{\varphi}\left(  h\right)  $ in (\ref{eqB})
together with (\ref{eqApp1}) and (\ref{eqId2}) shows that
\[
B_{\varphi}\left(  h\right)  =\frac{\Phi_{\left(  p+\lambda_{0},p+\lambda
_{1},-\lambda_{0},-\lambda_{1}\right)  }\left(  h\right)  }{2e^{\left(
\lambda_{0}+\lambda_{1}\right)  h}\Phi_{\left(  -2\lambda_{0},-2\lambda
_{1},-\lambda_{0}-\lambda_{1},p\right)  }\left(  h\right)  }=\frac
{\Phi_{\left(  p+\lambda_{0},p+\lambda_{1},-\lambda_{0},-\lambda_{1}\right)
}\left(  h\right)  }{2\Phi_{\left(  \lambda_{1}-\lambda_{0},\lambda
_{0}-\lambda_{1},0,p+\lambda_{0}+\lambda_{1}\right)  }\left(  h\right)  }%
\]
where we used (\ref{eqTrans}). Using (\ref{eqMinus}) we see that $T_{\left(
\lambda_{0},\lambda_{1}\right)  }^{\left(  p\right)  }\left(  -h\right)
=B\left(  h\right)  $ for $h>0$.

Note that $\Phi_{\left(  \lambda_{0},\lambda_{1}\right)  }\left(  t\right)
e^{pt}=\Phi_{\left(  p+\lambda_{0},p+\lambda_{1}\right)  }\left(  t\right)  $
by (\ref{eqTrans}), then we see using (\ref{eqInt}) and (\ref{eqId1}) in the
appendix
\[
C_{\varphi}\left(  h\right)  =\Phi_{\left(  \lambda_{0},\lambda_{1}\right)
}\left(  h\right)  \frac{\int_{0}^{h}\Phi_{\left(  \lambda_{0},\lambda
_{1}\right)  }\left(  t\right)  e^{pt}dt}{\int_{t_{0}}^{h}\Phi_{\left(
\lambda_{0},\lambda_{1}\right)  }\left(  t\right)  ^{2}e^{pt}dt}=\frac
{\Phi_{\left(  \lambda_{0},\lambda_{1}\right)  }\left(  h\right)
\Phi_{\left(  p+\lambda_{0},p+\lambda_{1},0\right)  }\left(  h\right)
}{2e^{pt}\Phi_{\left(  2\lambda_{0},2\lambda_{1},\lambda_{0}+\lambda
_{1},-p\right)  }\left(  h\right)  }.
\]
Multiply the nominator and denominator with $e^{-\left(  \lambda_{0}%
+\lambda_{1}\right)  h}$ and use (\ref{eqTrans}) to the nominator and
denominator:
\[
C_{\varphi}\left(  h\right)  =\frac{1}{2}\frac{\Phi_{\left(  -\lambda
_{0},-\lambda_{1}\right)  }\left(  h\right)  \Phi_{\left(  \lambda_{0}%
,\lambda_{1},-p\right)  }\left(  h\right)  }{\Phi_{\left(  \lambda_{0}%
-\lambda_{1},\lambda_{1}-\lambda_{0},0,-p-\lambda_{0}-\lambda_{1}\right)
}\left(  h\right)  }=S_{\left(  \lambda_{0},\lambda_{1}\right)  }^{\left(
p\right)  }\left(  h\right)  .
\]
Further the identities (\ref{eqId3}) and (\ref{eqId2}) in the appendix and
(\ref{eqMinus}) show that
\[
D_{\varphi}\left(  h\right)  =\Phi_{\left(  \lambda_{0},\lambda_{1}\right)
}\left(  -h\right)  \frac{\int_{0}^{h}\Phi_{\left(  \lambda_{0},\lambda
_{1}\right)  }\left(  t-h\right)  e^{pt}dt}{\int_{0}^{h}\Phi_{\left(
\lambda_{0},\lambda_{1}\right)  }\left(  t-h\right)  ^{2}e^{pt}dt}=\frac
{\Phi_{\left(  \lambda_{0},\lambda_{1}\right)  }\left(  -h\right)
\Phi_{\left(  -\lambda_{0},-\lambda_{1},p\right)  }\left(  h\right)  }%
{2\Phi_{\left(  -2\lambda_{0},-2\lambda_{1},-\lambda_{0}-\lambda_{1},p\right)
}\left(  h\right)  }.
\]
Using similar arguments it follows that the last expression is equal to
$S_{\left(  \lambda_{0},\lambda_{1}\right)  }^{\left(  p\right)  }\left(
-h\right)  .$
\end{proof}

\begin{corollary}
Let $\lambda_{0,j},\lambda_{1,j}$ be real and define $\varphi_{j}\left(
t\right)  =\Phi_{\left(  \lambda_{0,j},\lambda_{1,j}\right)  }\left(
t\right)  $ for $j=1,...,n-1,$ and $p$ be real number. Then for any $\eta>0$
there exists $\delta>0$ such that for any partition $t_{1}<...<t_{n}$ with
$t_{j+1}-t_{j}\leq\delta$ the matrix $\left(  s_{i,j}\right)  _{i,j=1,...,n}$
defined by $s_{i,j}=\left\langle H_{i},H_{j}\right\rangle _{p}$ is row
diagonally dominant, and
\[
\left\Vert P^{\mathcal{H}_{n}}\right\Vert \leq6+\eta.
\]

\end{corollary}

\begin{proof}
Using the rule of L'H\^{o}pital we see that for $h\rightarrow0$
\[
T_{\left(  \lambda_{0},\lambda_{1}\right)  }^{\left(  p\right)  }\left(
h\right)  =\frac{\Phi_{\left(  \lambda_{0},\lambda_{1},-p-\lambda
_{0},-p-\lambda_{1}\right)  }\left(  h\right)  }{2\Phi_{\left(  \lambda
_{0}-\lambda_{1},\lambda_{1}-\lambda_{0},0,-p-\lambda_{0}-\lambda_{1}\right)
}\left(  h\right)  }\rightarrow\frac{1}{2}.
\]
Using Leibniz's rule and the fact that $\Phi_{\left(  \lambda_{0},\lambda
_{1}\right)  }\left(  0\right)  =0$ and $\Phi_{\left(  \lambda_{0},\lambda
_{1},-p\right)  }\left(  0\right)  =\Phi_{\left(  \lambda_{0},\lambda
_{1},-p\right)  }^{\prime}\left(  0\right)  =0$ we have
\[
\lim_{t\rightarrow0}\frac{d^{3}}{dt^{3}}\left(  \Phi_{\left(  \lambda
_{0},\lambda_{1}\right)  }\left(  t\right)  \Phi_{\left(  \lambda_{0}%
,\lambda_{1},-p\right)  }\left(  t\right)  \right)  =3\Phi_{\left(
\lambda_{0},\lambda_{1}\right)  }^{\prime}\left(  0\right)  \Phi_{\left(
\lambda_{0},\lambda_{1},-p\right)  }^{\prime\prime}\left(  0\right)  =3.
\]
Since $\Phi_{\left(  2\lambda_{0},2\lambda_{1},\lambda_{0}+\lambda
_{1},-p\right)  }^{\left(  3\right)  }\left(  0\right)  =1$ we conclude
\[
\lim_{t\rightarrow0}S_{\left(  \lambda_{0},\lambda_{1}\right)  }^{\left(
p\right)  }\left(  t\right)  =\frac{3}{2}.
\]
Now we apply Theorem \ref{ThmMain1}: Take $\delta>0$ so small such that
$\varphi_{j}\left(  t\right)  $ are increasing on $\left[  -\delta
,\delta\right]  .$ Then $\sum_{j=1}^{n}\left\vert H_{j}\left(  t\right)
\right\vert \leq2$ by Proposition \ref{PropHatBasics}. By choosing $\delta>0$
small enough we may assume that $\left\vert S_{\left(  \lambda_{0,j}%
,\lambda_{1,j}\right)  }^{\left(  p\right)  }\left(  h\right)  \right\vert
\leq\varepsilon+3/2$ and $c\left(  \delta\right)  \leq\varepsilon+\frac{1}{2}$
for all $0\leq\left\vert h\right\vert \leq\delta$ and $j=1,...,n-1$ where
$c\left(  \delta\right)  $ is defined in (\ref{eqFormulac}). Then $1-c\left(
\delta\right)  \geq\frac{1}{2}-\varepsilon$ and (\ref{eqoperatornorm2}) shows
that
\[
\left\Vert P^{\mathcal{H}_{n}}\right\Vert _{\text{op}}\leq\frac{2}{\frac{1}%
{2}-\varepsilon}\left(  \frac{3}{2}+\frac{1}{2}\varepsilon\right)  .
\]
The right hand side of the last equation converges to $6$ for $\varepsilon
\rightarrow0,$ and now the claim is immediate.
\end{proof}

The following lemma is well known and we include the simple proof for completenessL\ 

\begin{lemma}
\label{LemPos1}Assume that $f:\mathbb{R}\rightarrow\mathbb{R}$ is
differentiable. If $D_{\lambda}f\left(  t\right)  >0$ for $t>a$ and $f\left(
a\right)  \geq0$ then $f\left(  t\right)  $ is positive for all $t>a.$ If
$D_{\lambda}f\left(  t\right)  <0$ for $t<a$ and $f\left(  a\right)  \leq0$
then $f\left(  t\right)  $ is negative for all $t<a.$
\end{lemma}

\begin{proof}
Define $g\left(  t\right)  =e^{-\lambda t}f\left(  t\right)  .$ Then
$g^{\prime}\left(  t\right)  =e^{-\lambda t}D_{\lambda}f\left(  t\right)  >0$
for $t>a.$ Hence $g$ is strictly increasing for $t>a,$ and since $g\left(
a\right)  =e^{-\lambda a}f\left(  a\right)  \geq0$ we infer that $g\left(
t\right)  >0.$ The second statement is proven in a similar way.
\end{proof}

\begin{theorem}
\label{ThmSbound}Assume that $\lambda_{0}\leq\lambda_{1}$ are real numbers.
Then $S_{\left(  \lambda_{0},\lambda_{1}\right)  }^{\left(  0\right)  }\left(
t\right)  \leq2$ for all real numbers $t.$
\end{theorem}

\begin{proof}
We assume that $\lambda_{0}<\lambda_{1},$ and at first we include the case of
a real number $p.$ In order to prove that $S_{\left(  \lambda_{0},\lambda
_{1}\right)  }^{\left(  p\right)  }\left(  t\right)  \leq C$ for all $t>0$ it
suffices to show
\[
F_{p}\left(  t\right)  :=2C\Phi_{\left(  \lambda_{0}-\lambda_{1},\lambda
_{1}-\lambda_{0},0,-p-\lambda_{0}-\lambda_{1}\right)  }\left(  t\right)
-\Phi_{\left(  -\lambda_{0},-\lambda_{1}\right)  }\left(  t\right)
\Phi_{\left(  \lambda_{0},\lambda_{1},-p\right)  }\left(  t\right)
\]
is non-negative for $t>0.$ Using (\ref{eqTrans}) we can rewrite
\begin{align*}
\Phi_{\left(  -\lambda_{0},-\lambda_{1}\right)  }\left(  t\right)
\Phi_{\left(  \lambda_{0},\lambda_{1},-p\right)  }\left(  t\right)   &
=\frac{e^{-\lambda_{0}t}-e^{-\lambda_{1}t}}{\lambda_{1}-\lambda_{0}}%
\Phi_{\left(  \lambda_{0},\lambda_{1},-p\right)  }\left(  t\right) \\
&  =\frac{\Phi_{\left(  0,\lambda_{1}-\lambda_{0},-\lambda_{0}-p\right)
}\left(  t\right)  -\Phi_{\left(  \lambda_{0}-\lambda_{1},0,-\lambda
_{1}-p\right)  }\left(  t\right)  }{\lambda_{1}-\lambda_{0}}.
\end{align*}
Using (\ref{rec0}) in the appendix it follows that%
\[
F_{p}^{\prime}=2C\Phi_{\left(  \lambda_{0}-\lambda_{1},\lambda_{1}-\lambda
_{0},-p-\lambda_{0}-\lambda_{1}\right)  }-\frac{\Phi_{\left(  \lambda
_{1}-\lambda_{0},-\lambda_{0}-p\right)  }-\Phi_{\left(  \lambda_{0}%
-\lambda_{1},-\lambda_{1}-p\right)  }}{\lambda_{1}-\lambda_{0}}.
\]
Next we consider
\[
G_{p}:=\left(  \lambda_{1}-\lambda_{0}\right)  D_{\left(  -p-\lambda
_{0}-\lambda_{1}\right)  }F_{p}^{\prime}=2C\left(  \lambda_{1}-\lambda
_{0}\right)  \Phi_{\left(  \lambda_{0}-\lambda_{1},\lambda_{1}-\lambda
_{0}\right)  }-R
\]
where we define%
\[
R=D_{\left(  -p-\lambda_{0}-\lambda_{1}\right)  }\Phi_{\left(  \lambda
_{1}-\lambda_{0},-\lambda_{0}-p\right)  }-D_{\left(  -p-\lambda_{0}%
-\lambda_{1}\right)  }\Phi_{\left(  \lambda_{0}-\lambda_{1},-\lambda
_{1}-p\right)  }.
\]
Then
\begin{align*}
R  &  =D_{\left(  -p-\lambda_{0}\right)  }\Phi_{\left(  \lambda_{1}%
-\lambda_{0},-\lambda_{0}-p\right)  }+\lambda_{1}\Phi_{\left(  \lambda
_{1}-\lambda_{0},-\lambda_{0}-p\right)  }\\
&  -D_{\left(  -p-\lambda_{1}\right)  }\Phi_{\left(  \lambda_{0}-\lambda
_{1},-\lambda_{1}-p\right)  }-\lambda_{0}\Phi_{\left(  \lambda_{0}-\lambda
_{1},-\lambda_{1}-p\right)  }\\
&  =\Phi_{\left(  \lambda_{1}-\lambda_{0}\right)  }-\Phi_{\left(  \lambda
_{0}-\lambda_{1}\right)  }+\lambda_{1}\Phi_{\left(  \lambda_{1}-\lambda
_{0},-\lambda_{0}-p\right)  }-\lambda_{0}\Phi_{\left(  \lambda_{0}-\lambda
_{1},-\lambda_{1}-p\right)  }.
\end{align*}
Since $\Phi_{\left(  \lambda_{0}-\lambda_{1},\lambda_{1}-\lambda_{0}\right)
}\left(  t\right)  =\frac{e^{\left(  \lambda_{1}-\lambda_{0}\right)
t}-e^{-\left(  \lambda_{1}-\lambda_{0}\right)  t}}{2\left(  \lambda
_{1}-\lambda_{0}\right)  }$ we obtain the following formula:
\begin{align*}
G_{p}\left(  t\right)   &  =\left(  C-1\right)  e^{\left(  \lambda_{1}%
-\lambda_{0}\right)  t}-\left(  C-1\right)  e^{\left(  \lambda_{0}-\lambda
_{1}\right)  t}\\
&  -\lambda_{1}\Phi_{\left(  \lambda_{1}-\lambda_{0},-\lambda_{0}-p\right)
}\left(  t\right)  +\lambda_{0}\Phi_{\left(  \lambda_{0}-\lambda_{1}%
,-\lambda_{1}-p\right)  }\left(  t\right)  .
\end{align*}
We specialize to the case $p=0$, and obtain the equality:
\begin{align*}
G_{0}\left(  t\right)   &  =\left(  C-1\right)  e^{\left(  \lambda_{1}%
-\lambda_{0}\right)  t}-\left(  C-1\right)  e^{-\left(  \lambda_{1}%
-\lambda_{0}\right)  t}\\
&  -\lambda_{1}\frac{e^{\left(  \lambda_{1}-\lambda_{0}\right)  t}%
-e^{-\lambda_{0}t}}{\lambda_{1}}+\lambda_{0}\frac{e^{\left(  \lambda
_{0}-\lambda_{1}\right)  t}-e^{-\lambda_{1}t}}{\lambda_{0}}\\
&  =\left(  C-2\right)  e^{\left(  \lambda_{1}-\lambda_{0}\right)  t}-\left(
C-2\right)  e^{-\left(  \lambda_{1}-\lambda_{0}\right)  t}+e^{-\lambda_{0}%
t}-e^{-\lambda_{1}t}.\\
&  =\left(  C-2\right)  \left(  \lambda_{1}-\lambda_{0}\right)  \Phi_{\left(
\lambda_{1}-\lambda_{0},\lambda_{1}-\lambda_{0}\right)  }+\left(  \lambda
_{1}-\lambda_{0}\right)  \Phi_{\left(  -\lambda_{0},-\lambda_{1}\right)
}\left(  t\right)  .
\end{align*}
If we take $C=2$ in this equation we see that $G_{0}\left(  t\right)  =\left(
\lambda_{1}-\lambda_{0}\right)  \Phi_{\left(  -\lambda_{0},-\lambda
_{1}\right)  }\left(  t\right)  >0$ for $t>0$. Thus $G_{0}\left(  t\right)  $
is positive for $t>0$ and $G_{0}\left(  0\right)  =0.$ Lemma \ref{LemPos1}
shows that $F_{0}\left(  t\right)  \geq0$ for all $t>0$, hence $S_{\left(
\lambda_{0},\lambda_{1}\right)  }^{\left(  0\right)  }\left(  t\right)  \leq2$
for all $t>0.$

For $t<0$ we see that $G_{0}\left(  t\right)  <0$ and Lemma \ref{LemPos1}
shows that $F_{0}^{\prime}\left(  t\right)  \leq0$ for all $t<0.$ Further
$F_{0}^{\prime}\left(  0\right)  =0,$ so $F_{0}\left(  t\right)  \leq0$ for
all $t<0.$ Since both $\Phi_{\left(  \lambda_{0}-\lambda_{1},\lambda
_{1}-\lambda_{0},0,-p-\lambda_{0}-\lambda_{1}\right)  }\left(  t\right)  $ and
$\Phi_{\left(  -\lambda_{0},-\lambda_{1}\right)  }\left(  t\right)  $ are
negative for $t<0$ and $\Phi_{\left(  \lambda_{0},\lambda_{1},-p\right)
}\left(  t\right)  >0$ for $t<0,$ it follows that $S_{\left(  \lambda
_{0},\lambda_{1}\right)  }^{\left(  0\right)  }\left(  t\right)  \leq2$ for
$t<0.$

The case $\lambda_{0}=\lambda_{1}$ follows from a continuity argument in the
variables $\lambda_{0},\lambda_{1}.$
\end{proof}

\begin{theorem}
\label{CorLebsgueSym}Let $\xi_{j}$ be real numbers and define $\varphi
_{j}\left(  t\right)  =\Phi_{\left(  \xi_{j},-\xi_{j}\right)  }\left(
t\right)  $ for $j=1,...,n-1.$Then the matrix $\left(  s_{i,j}\right)
_{i,j=1,...,n}$ defined by $s_{i,j}=\left\langle H_{i},H_{j}\right\rangle
_{0}$ is row diagonally dominant, and
\[
\left\Vert P^{\mathcal{H}_{n}}\right\Vert _{\text{op}}\leq4.
\]

\end{theorem}

\begin{proof}
By Proposition \ref{PropSumHat}, $\sum_{j=1}^{n}\left\vert H_{j}\left(
t\right)  \right\vert \leq1,$ and Theorem \ref{ThmSbound} shows that
$S_{\left(  \lambda_{0},\lambda_{1}\right)  }^{\left(  0\right)  }\left(
t\right)  \leq2$ for all real $t.$ Hence according to Theorem \ref{ThmMain1}
\[
\left\Vert P^{\mathcal{H}_{n}}\right\Vert _{\text{op}}\leq\frac{2}{1-c}.
\]
It suffices to show that for any real $\xi$ and real $t$
\[
T_{\left(  \xi,-\xi\right)  }^{\left(  0\right)  }\left(  t\right)  =\frac
{1}{2}\frac{\Phi_{\left(  \xi,-\xi,\xi,-\xi\right)  }}{\Phi_{\left(
2\xi,-2\xi,0,0\right)  }}\left(  t\right)  \leq\frac{1}{2}\text{.}%
\]
It is easily verified that
\[
\Phi_{\left(  1,-1,1,-1\right)  }\left(  t\right)  =\frac{1}{2}\left(  t\cosh
t-\sinh t\right)  \text{ and }\Phi_{\left(  2,-2,0,0\right)  }\left(
t\right)  =\frac{1}{8}\left(  \sinh\left(  2t\right)  -2t\right)  .
\]
Since $\Phi_{\xi\Lambda_{N}}\left(  t\right)  =\frac{1}{\xi^{N}}\Phi
_{\Lambda_{N}}\left(  \xi t\right)  $ it follows that
\[
0\leq\frac{\Phi_{\left(  \xi,-\xi,\xi,-\xi\right)  }}{\Phi_{\left(  2\xi
,-2\xi,0,0\right)  }}\left(  t\right)  =\frac{\Phi_{\left(  1,-1,1,-1\right)
}}{\Phi_{\left(  2,-2,0,0\right)  }}\left(  \xi t\right)  \text{ and }%
\frac{4\left(  t\cosh t-\sinh t\right)  }{\sinh\left(  2t\right)  -2t}\leq1
\]
The proof is complete.
\end{proof}

\begin{remark}
In the polynomial case (i.e. $\xi=0)$ the estimate of the operator norm can be
improved to $\left\Vert P^{\mathcal{H}_{n}}\right\Vert _{\text{op}}\leq3$
since in this case
\[
S_{\left(  0,0\right)  }^{\left(  0\right)  }\left(  h\right)  =\frac{1}%
{2}\frac{\Phi_{\left(  0,0\right)  }\left(  h\right)  \Phi_{\left(
0,0,0\right)  }\left(  h\right)  }{\Phi_{\left(  0,0,0,0\right)  }\left(
h\right)  }=\frac{1}{2}\frac{h\left(  \frac{1}{2}h^{2}\right)  }{\frac{1}%
{6}h^{3}}=\frac{3}{2}.
\]

\end{remark}

For tension splines one has the following result:

\begin{theorem}
For $\left(  \lambda_{0},\lambda_{1}\right)  =\left(  0,\rho\right)  $ and
$p=0$ the matrix $\left(  s_{i,j}\right)  _{i,j=1,...,n}$ defined by
$s_{i,j}=\left\langle H_{i},H_{j}\right\rangle _{0}$ is row diagonally
dominant, and
\[
\left\Vert P^{\mathcal{H}_{n}}\right\Vert _{\text{op}}<\infty.
\]

\end{theorem}

\begin{proof}
The matrix $\left(  s_{i,j}\right)  _{i,j=1,...,n}$ is row diagonally dominant
if
\[
T_{\left(  0,\rho\right)  }^{\left(  0\right)  }\left(  t\right)  =\frac{1}%
{2}\frac{\Phi_{\left(  0,\rho,0,-\rho\right)  }}{\Phi_{\left(  -\rho
,\rho,0,\rho\right)  }}\left(  t\right)  =\frac{1}{2}\frac{\Phi_{\left(
0,1,0,-1\right)  }}{\Phi_{\left(  -1,1,0,1\right)  }}\left(  \rho t\right)
\]
is smaller than $1.$ Since $\Phi_{\left(  0,1,0,-1\right)  }\left(  t\right)
=\sinh t-t$ and
\[
\Phi_{\left(  -1,1,0,1\right)  }\left(  t\right)  =\frac{1}{2}\left(
e^{t}\left(  t-2\right)  +\sinh t+2\right)
\]
one obtains that
\[
\frac{\Phi_{\left(  0,1,0,-1\right)  }\left(  t\right)  }{\Phi_{\left(
-1,1,0,1\right)  }\left(  t\right)  }=\frac{2\left(  \sinh t-t\right)  }%
{e^{t}\left(  t-2\right)  +\sinh t+2}=\frac{2\left(  \frac{1}{2}e^{2t}%
-\frac{1}{2}-te^{t}\right)  }{e^{2t}\left(  t-2\right)  +\frac{1}{2}%
e^{2t}-\frac{1}{2}+2e^{t}}\rightarrow2
\]
for $t\longmapsto-\infty.$
\end{proof}

It is a natural question whether one has similar results for the general case
$\lambda_{0}<\lambda_{1}$. For positive frequencies $\left(  \lambda
_{0},\lambda_{1}\right)  =\left(  2,1\right)  $ one obtains
\[
T_{\left(  2,1\right)  }^{\left(  0\right)  }\left(  t\right)  =\frac{1}%
{2}\frac{\Phi_{\left(  2,1,-2,-1\right)  }\left(  t\right)  }{\Phi_{\left(
1,-1,0,3\right)  }\left(  t\right)  }=\frac{4\left(  \sinh t-\frac{1}{2}%
\sinh2t\right)  }{-e^{3t}+3\sinh t+9\cosh t-9+1}%
\]
and the graph of $T_{\left(  2,1\right)  }^{\left(  0\right)  }\left(
t\right)  $ is provided by Figure $7$.%


\begin{figure}
[h] 
\begin{center}
\includegraphics[
height=2.1423in,
width=4.1945in
]%
{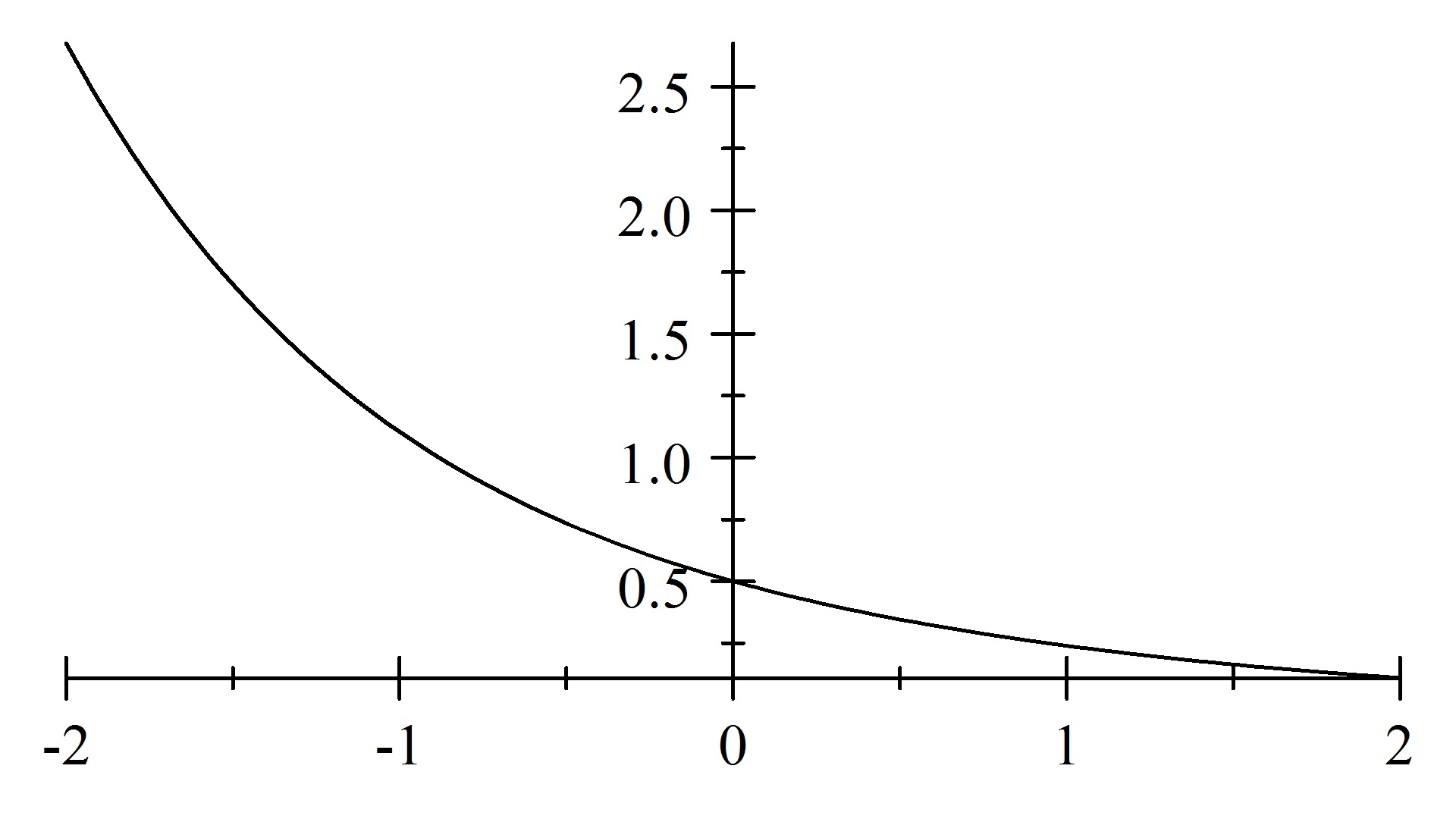}%
\caption{Graph of $T_{\left(  2,1\right)  }^{\left(  0\right)  }$}%
\end{center}
\end{figure}

Hence the constant $c\left(  \delta\right)  $ in unbounded in the variable
$\delta.$ In the case $\left(  \lambda_{0},\lambda_{1}\right)  =\left(
-2,1\right)  $ one obtains that
\[
T_{\left(  -2,1\right)  }^{\left(  0\right)  }\left(  t\right)  =\frac{4}%
{3}\frac{\sinh t-\frac{1}{2}\sinh2t}{-e^{-t}-\frac{1}{3}\sinh3t+\frac{1}%
{9}\cosh3t+\frac{8}{9}}%
\]
In this case the function $t\longmapsto T_{\left(  -2,1\right)  }^{\left(
0\right)  }\left(  t\right)  $ is not symmetric and it does not attain its
maximum at $t=0$ as seen by its graph provided on Figure $8$.

%


\begin{figure}
[h] 
\begin{center}
\includegraphics[
height=2.1423in,
width=4.1945in
]
{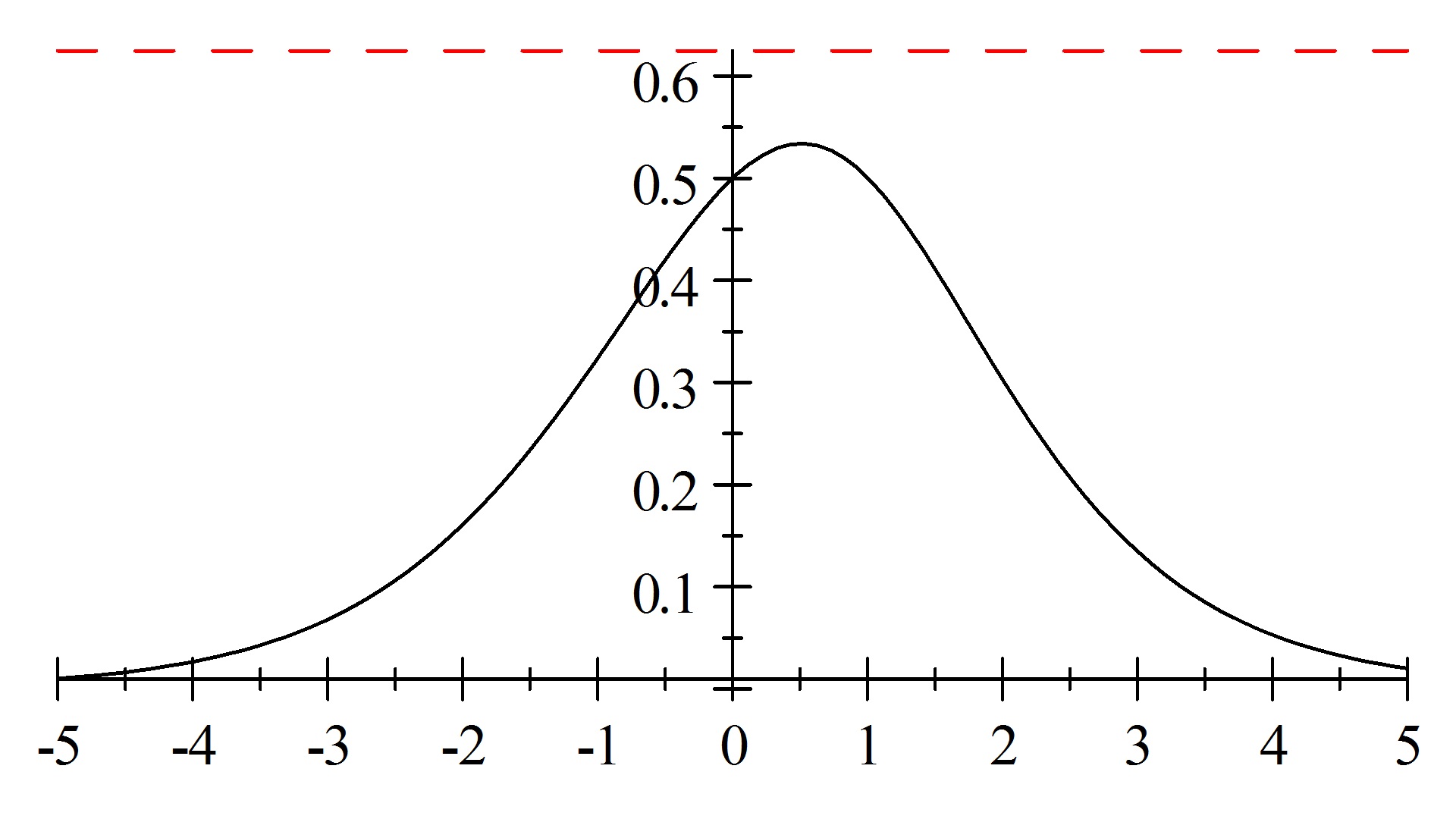}%
\caption{Graph of $T_{\left(  -2,1\right)  }^{\left(  0\right)  }$ and
its upper bound}%
\end{center}
\end{figure}

The following theorem gives an upper bound for $T_{\left(  \lambda_{0}%
,\lambda_{1}\right)  }^{\left(  0\right)  }\left(  t\right)  $:

\begin{theorem}
\label{ThmMain2}For $\lambda_{0}<0<\lambda_{1}$ the following inequality holds
for all real $t$
\begin{equation}
T_{\left(  \lambda_{0},\lambda_{1}\right)  }^{\left(  0\right)  }\left(
t\right)  =\frac{1}{2}\frac{\Phi_{\left(  \lambda_{0},\lambda_{1},-\lambda
_{0},-\lambda_{1}\right)  }\left(  t\right)  }{\Phi_{\left(  \lambda
_{0}-\lambda_{1},\lambda_{1}-\lambda_{0},0,-\lambda_{0}-\lambda_{1}\right)
}\left(  t\right)  }\leq\max\left\{  \frac{2\lambda_{1}-\lambda_{0}}%
{2\lambda_{1}-4\lambda_{0}},\frac{\lambda_{1}-2\lambda_{0}}{4\lambda
_{1}-2\lambda_{0}}\right\}  <1. \label{ineqTTT}%
\end{equation}

\end{theorem}

\begin{proof}
According to Lemma \ref{Lem1App} in the appendix, for any $a<0$ the function
\[
F_{a,-a-1}\left(  t\right)  =2C_{a,-a-1}\Phi_{\left(  a-1,1-a,0,-a-1\right)
}\left(  t\right)  -\Phi_{\left(  a,1,-a,-1\right)  }\left(  t\right)
\]
is positive for $t>0$ where
\[
C_{a,-a-1}=\max\left\{  \frac{1}{2},\frac{3}{2\left(  2-a\right)  }%
,\frac{3-4a}{2\left(  1-2a\right)  \left(  2-a\right)  },\frac{a+2}{2\left(
1-2a\right)  }\right\}  .
\]
It is easy to see that
\[
C_{a,-a-1}\leq\max\left\{  \frac{1}{2},\frac{2-a}{2-4a}\right\}  .
\]
Let us put $a=\lambda_{0}/\lambda_{1}<0$ and $b=-a-1$. Then
\[
F_{a,-a-1}\left(  \lambda_{1}t\right)  =2C_{a,-a-1}\Phi_{a-1,1-a,0,-a-1}%
\left(  \lambda_{1}t\right)  -\Phi_{\left(  a,-a,1,-1\right)  }\left(
\lambda_{1}t\right)
\]
is positive and therefore
\[
\frac{\Phi_{\left(  \lambda_{0},\lambda_{1},-\lambda_{0},-\lambda_{1}\right)
}\left(  t\right)  }{\Phi_{\left(  \lambda_{0}-\lambda_{1},\lambda_{1}%
-\lambda_{0},0,-\lambda_{0}-\lambda_{1}\right)  }\left(  t\right)  }%
=\frac{\Phi_{\left(  a,-a,1,-1\right)  }\left(  \lambda_{1}t\right)  }%
{\Phi_{a-1,1-a,0,-a-1}\left(  \lambda_{1}t\right)  }\leq2C_{a,-a-1}%
\]
Clearly this implies that $T_{\left(  \lambda_{0},\lambda_{1}\right)
}^{\left(  0\right)  }\left(  t\right)  \leq C_{a,-a-1}$ for all $t>0.$

Now put $a=\lambda_{1}/\lambda_{0}<0$ and $b=-1-a.$ Then $F_{a,-a-1}\left(
t\right)  $ is positive for $t>0$. Since $-\lambda_{0}t>0$ we infer that
\begin{equation}
F_{a,-a-1}\left(  -\lambda_{0}t\right)  =2C_{a,-a-1}\Phi_{\left(
a-1,1-a,0,-a-1\right)  }\left(  -\lambda_{0}t\right)  -\Phi_{\left(
a,1,-a,-1\right)  }\left(  -\lambda_{0}t\right)  \label{eqF2}%
\end{equation}
is positive for all $t>0.$ Hence
\[
\frac{\Phi_{\left(  \lambda_{0},\lambda_{1},-\lambda_{0},-\lambda_{1}\right)
}\left(  -t\right)  }{\Phi_{\left(  \lambda_{1}-\lambda_{0},\lambda
_{0}-\lambda_{1},0,-\lambda_{1}-\lambda_{0}\right)  }\left(  -t\right)
}=\frac{\Phi_{\left(  a,1,-a,-1\right)  }\left(  \lambda_{0}\left(  -t\right)
\right)  }{\Phi_{\left(  a-1,1-a,0,-a-1\right)  }\left(  \lambda_{0}\left(
-t\right)  \right)  }\leq2C_{a,-a-1}%
\]
It follows that $T_{\left(  \lambda_{0},\lambda_{1}\right)  }^{\left(
0\right)  }\left(  t\right)  \leq C$ for all $t\in\mathbb{R}$ where
\[
C\leq\max\left\{  C_{a,-a-1}:a=\frac{\lambda_{0}}{\lambda_{1}}\text{ or
}a=\frac{\lambda_{1}}{\lambda_{0}}.\right\}
\]
Hence
\[
C\leq\max\left\{  \frac{1}{2},\frac{2-\frac{\lambda_{0}}{\lambda_{1}}%
}{2-4\frac{\lambda_{0}}{\lambda_{1}}},\frac{2-\frac{\lambda_{1}}{\lambda_{0}}%
}{2-4\frac{\lambda_{1}}{\lambda_{0}}}\right\}  =\max\left\{  \frac
{2\lambda_{1}-\lambda_{0}}{2\lambda_{1}-4\lambda_{0}},\frac{\lambda
_{1}-2\lambda_{0}}{4\lambda_{1}-2\lambda_{0}}\right\}  .
\]

\end{proof}

Now we specify the general result of Theorem \ref{ThmMain1}.

\begin{theorem}
Let $\lambda_{0j}<0<\lambda_{1,j}$ be real numbers and define $\varphi
_{j}=\Phi_{\left(  \lambda_{0,j},\lambda_{1,j}\right)  }$ for $j=1,...,n-1.$
Then the matrix $\left(  s_{i,j}\right)  _{i,j=1,...,n}$ defined by
$s_{i,j}=\left\langle H_{i},H_{j}\right\rangle _{0}$ is row diagonally
dominant, and
\[
\left\Vert P^{\mathcal{H}_{n}}\right\Vert _{\text{op}}\leq2\max_{j=1,...,n-1}%
\max\left\{  \frac{2\lambda_{1,j}-4\lambda_{0,j}}{-3\lambda_{0j}}%
,\frac{4\lambda_{1,j}-2\lambda_{0,j}}{3\lambda_{1,j}}\right\}
\]

\end{theorem}

\begin{proof}
By Proposition \ref{PropSumHat}, $\sum_{j=1}^{n}\left\vert H_{j}\left(
t\right)  \right\vert \leq1,$ and Theorem \ref{ThmSbound} shows that
$S_{\left(  \lambda_{0},\lambda_{1}\right)  }^{\left(  0\right)  }\left(
t\right)  \leq2$ for all real $t.$ Hence according to Theorem \ref{ThmMain1}
\[
\left\Vert P^{\mathcal{H}_{n}}\right\Vert _{\text{op}}\leq\frac{2}{1-c}.
\]
Theorem \ref{ThmMain2} shows that
\[
c\leq\max_{j=1,...,n-1}\max\left\{  \frac{2\lambda_{1,j}-\lambda_{0,j}%
}{2\lambda_{1,j}-4\lambda_{0,j}},\frac{\lambda_{1,j}-2\lambda_{0,j}}%
{4\lambda_{1,j}-2\lambda_{0,j}}\right\}  <1.
\]
It follows that for some $j=1,...,n,$ at least one of the following two
inequalities holds:
\begin{align*}
1-c  &  \geq1-\frac{2\lambda_{1,j}-\lambda_{0,j}}{2\lambda_{1,j}%
-4\lambda_{0,j}}=\frac{-3\lambda_{0,j}}{2\lambda_{1,j}-4\lambda_{0,j}}\\
1-c  &  \geq1-\frac{\lambda_{1,j}-2\lambda_{0,j}}{4\lambda_{1,j}%
-2\lambda_{0,j}}=\frac{3\lambda_{1,j}}{4\lambda_{1,j}-2\lambda_{0,j}}%
\end{align*}
It follows that
\[
\frac{1}{1-c}\leq\max_{j=1,...,n-1}\max\left\{  \frac{2\lambda_{1,j}%
-4\lambda_{0,j}}{-3\lambda_{0j}},\frac{4\lambda_{1,j}-2\lambda_{0,j}}%
{3\lambda_{1,j}}\right\}  .
\]

\end{proof}

\section{\label{S7}Error estimates for interpolation with exponential splines
of order $4$}

Assume that $f,g:\left(  t_{j-1},t_{j}\right)  \rightarrow\mathbb{R}$ are
differentiable and that $f\left(  t\right)  ,g\left(  t\right)  $ have limits
when $t\rightarrow t_{j},t_{j+1}$ for $t\in\left(  t_{j},t_{j+1}\right)  ,$
and let $D_{\lambda}f=f^{\prime}-\lambda f.$ Partial integration shows that
\begin{equation}
\int_{t_{j}}^{t_{j+1}}D_{\lambda_{2,j}}f\left(  t\right)  \cdot g\left(
t\right)  dt=fg\mid_{t_{j}}^{t_{j+1}}-\int_{t_{j}}^{t_{j+1}}f\left(  t\right)
D_{-\lambda_{2,j}}g\left(  t\right)  dt \label{eqDpartial}%
\end{equation}
where $f\mid_{t_{j}}^{t_{j+1}}$ is defined as the difference $f\left(
t_{j+1}\right)  -f\left(  t_{j}\right)  .$ Note that
\begin{equation}
D_{-\lambda_{2,j}}\left(  h\left(  t\right)  e^{pt}\right)  =e^{pt}\left(
h^{\prime}\left(  t\right)  +ph\left(  t\right)  +\lambda_{2,j}h\left(
t\right)  \right)  =e^{pt}D_{-p-\lambda_{2,j}}h\left(  t\right)  .
\label{eqpp}%
\end{equation}
Take now $g=he^{pt}$ in formula (\ref{eqDpartial}) and use (\ref{eqpp}):
\begin{equation}
\int_{t_{j}}^{t_{j+1}}D_{\lambda_{2,j}}f\cdot he^{pt}dt=fhe^{pt}\mid_{t_{j}%
}^{t_{j+1}}-\int_{t_{j}}^{t_{j+1}}fD_{-p-\lambda_{2,j}}h\cdot e^{pt}dt.
\label{eqPI3}%
\end{equation}
Now replace $f$ by $D_{\lambda_{3,j}}f,$ and we see that
\begin{align*}
\int_{t_{j}}^{t_{j+1}}D_{\lambda_{2,j}}D_{\lambda_{3,j}}f\cdot he^{pt}dt  &
=D_{\lambda_{3,j}}f\cdot he^{pt}\mid_{t_{j}}^{t_{j+1}}\\
&  -\int_{t_{j}}^{t_{j+1}}D_{\lambda_{3,j}}f\cdot D_{-p-\lambda_{2,j}\text{ }%
}h\cdot e^{pt}dt.
\end{align*}
We apply (\ref{eqPI3}) to the last summand with $\lambda_{3,j}$ instead of
$\lambda_{2,j}$ and replace $h$ by $D_{-p-\lambda_{2,j}}h,$ so we have the
general identity
\begin{align*}
\int_{t_{j}}^{t_{j+1}}D_{\lambda_{2,j}}D_{\lambda_{3,j}}f\cdot he^{pt}dt  &
=D_{\lambda_{3,j}}f\cdot he^{pt}\mid_{t_{j}}^{t_{j+1}}\\
-fD_{-p-\lambda_{2,j}}h\cdot e^{pt}  &  \mid_{t_{j}}^{t_{j+1}}+\int_{t_{j}%
}^{t_{j+1}}fD_{-p-\lambda_{3,j}}D_{-p-\lambda_{2,j}\text{ }}h\cdot e^{pt}dt.
\end{align*}

\begin{proposition}
\label{PropOrtho}Assume that for the real numbers $\lambda_{0,j},\lambda
_{1,j}$,$\lambda_{2,j}$,$\lambda_{3,j}$ for $j=1,...,n-1$ there exists a real
number $p$ such that for all $j=1,....,n-1$
\begin{equation}
\lambda_{0,j}=-p-\lambda_{2,j}\text{ and }\lambda_{1,j}=-p-\lambda_{3,j}.
\label{eqcondlambda}%
\end{equation}
Assume further that $F\in C^{2}\left[  t_{1},t_{n}\right]  $ vanishes on
$t_{1},...,t_{n}$ and $F^{\prime}\left(  t_{1}\right)  =0$ and $F^{\prime
}\left(  t_{n}\right)  =0.$ Define $f\in C\left[  t_{1},t_{n}\right]  $ by
setting $f\left(  t\right)  =D_{\lambda_{2,j}}D_{\lambda_{3,j}}F\left(
t\right)  $ for $t\in\left[  t_{j},t_{j+1}\right]  $ and $j=1,...,n-1.$ Then
\[
\left\langle f,h\right\rangle _{p}=0
\]
for any piecewise exponential spline $h$ with respect to $L_{\left(
\lambda_{0,j},\lambda_{1,j}\right)  }$ for $j=1,...,n-1.$
\end{proposition}

\begin{proof}
Since we assume that $F\left(  t_{j}\right)  =0$ for $j=1,...,n$ we obtain
\[
\left\langle f,h\right\rangle _{p}=\sum_{j=1}^{n-1}\int_{t_{j}}^{t_{j+1}%
}D_{\lambda_{2,j}}D_{\lambda_{3,j}}F\cdot he^{pt}dt=\sum_{j=1}^{n-1}%
D_{\lambda_{3,j}}f\cdot he^{pt}\mid_{t_{j}}^{t_{j+1}}%
\]
since the expressions $FD_{-p-\lambda_{2,j}}h\cdot e^{pt}\mid_{t_{j}}%
^{t_{j+1}}$clearly vanish and
\[
\int_{t_{j}}^{t_{j+1}}F\left(  t\right)  D_{-p-\lambda_{3,j}}D_{-p-\lambda
_{2,j}\text{ }}h\left(  t\right)  \cdot e^{pt}dt=0
\]
since $h\left(  t\right)  $ is an exponential spline $h$ with respect to
$L_{\left(  \lambda_{0,j},\lambda_{1,j}\right)  }$ for $j=1,...,n-1,$ and
condition (\ref{eqcondlambda}) is satisfied. Since $D_{\lambda_{3,j}}f\left(
t\right)  $ and $h\left(  t\right)  e^{pt}$ is continuous we infer that the
following sum is telescoping:
\[
\sum_{j=1}^{n-1}\left(  D_{\lambda_{3,j}}f\cdot he^{pt}\mid_{t_{j}}^{t_{j+1}%
}\right)  =D_{\lambda_{3,n-1}}F\left(  t_{n}\right)  h\left(  t_{n}\right)
e^{pt_{n}}-D_{\lambda_{3,1}}F\left(  t_{1}\right)  h\left(  t_{1}\right)
e^{pt_{1}}=0
\]
since $F\left(  t_{n}\right)  =F^{\prime}\left(  t_{n}\right)  =0$ and
$F\left(  t_{1}\right)  =F^{\prime}\left(  t_{1}\right)  =0.$
\end{proof}

\begin{theorem}
Assume that there exists a real number $p$ such that $-p-\lambda_{2,j}%
=\lambda_{0,j}$ and $-p-\lambda_{3,j}=\lambda_{1,j}$ for $j=1,...,n. $ Let
$t_{1}<...<t_{n}$ and $F\in C^{4}\left[  t_{1},t_{n}\right]  $. Assume that
$I_{4}\left(  F\right)  $ is a piecewise exponential spline with respect to
$L_{\left(  \lambda_{0,j},\lambda_{1,j},\lambda_{2,j},\lambda_{3,j}\right)  }$
for $j=1,...,n-1,$ interpolating $F$ at the points $t_{1},...,t_{n}$ and such
that
\[
\frac{d}{dt}F\left(  t_{1}\right)  =\frac{d}{dt}I_{4}\left(  F\right)  \left(
t_{1}\right)  \text{ and }\frac{d}{dt}F\left(  t_{n}\right)  =\frac{d}%
{dt}I_{4}\left(  F\right)  \left(  t_{n}\right)  .
\]
Then for any $t\in\left[  t_{j},t_{j+1}\right]  $ the estimate
\[
\left\vert F\left(  t\right)  -I_{4}\left(  F\right)  \left(  t\right)
\right\vert \leq C\cdot\max_{\zeta\in\left[  t_{j},t_{j+1}\right]
\ ,j=1,...,n-1}\left\vert L_{\left(  \lambda_{0,j},\lambda_{1,j},\lambda
_{2,j},\lambda_{3,j}\right)  }F\left(  \zeta\right)  \right\vert
\]
holds, where
\[
C=\left(  1+\left\Vert P^{\mathcal{H}_{n}}\right\Vert _{\text{op}}\right)
\max_{j=1,...n-1}M_{\lambda_{2,j},\lambda_{3,j}}^{t_{j},t_{j+1}}%
\max_{j=1,...n}M_{\lambda_{0,j}\lambda_{1,j}}^{t_{j}t_{j+1}}%
\]
and $\left\Vert P^{\mathcal{H}_{n}}\right\Vert _{\text{op}}$ is given in
(\ref{eqoperatornorm2}) with respect to the weight function $e^{pt}.$
\end{theorem}

\begin{proof}
The function $t\longmapsto$ $I_{4}\left(  F\right)  \left(  t\right)  $ for
$t\in\left[  t_{j},t_{j+1}\right]  $ is an exponential polynomial in $E\left(
\lambda_{0,j},\lambda_{1,j},\lambda_{2,j},\lambda_{3,j}\right)  .$ Further
$g\left(  t\right)  :=F\left(  t\right)  -I_{4}\left(  F\right)  \left(
t\right)  $ for $t\in\left[  t_{j},t_{j+1}\right]  $ vanishes in $t_{j}$ and
$t_{j+1}.$ Using inequality (\ref{eqmaxderiv}) for the function $g\left(
t\right)  $ and for the constants $\lambda_{2,j},$ $\lambda_{3,j},$ we see
that
\[
\left\vert F\left(  t\right)  -I_{4}\left(  F\right)  \left(  t\right)
\right\vert \leq M_{\lambda_{2,j},\lambda_{3,j}}^{t_{j},t_{j+1}}\max_{\zeta
\in\left[  t_{j},t_{j+1}\right]  }\left\vert D_{\lambda_{2,j}}D_{\lambda
_{3,j}}\left[  F-I_{4}\left(  F\right)  \right]  \left(  \zeta\right)
\right\vert
\]
for $t\in\left[  t_{j},t_{j+1}\right]  .$ We define $f\in C\left[  t_{1}%
,t_{n}\right]  $ and $f_{0}\in C\left[  t_{1},t_{n}\right]  $ by putting
\begin{align*}
f\left(  t\right)   &  =D_{\lambda_{2,j}}D_{\lambda_{3,j}}F\left(  t\right)
\quad\quad\text{ for }t\in\left[  t_{j},t_{j+1}\right] \\
f_{0}\left(  t\right)   &  =D_{\lambda_{2,j}}D_{\lambda_{3,j}}\left(
I_{4}\left(  F\right)  \right)  \left(  t\right)  \quad\quad\text{ for }%
t\in\left[  t_{j},t_{j+1}\right]  .
\end{align*}
From Proposition \ref{PropOrtho} we infer that $\left\langle f-f_{0}%
,h\right\rangle _{p}=0$ for any exponential spline $h$ with respect to
$L_{\left(  \lambda_{0,j},\lambda_{1,j}\right)  }$ for $j=1,...,n-1.$ Hence
$f_{0}$ is the best $L^{2}$-approximation to the function $f$ from the
subspace $\mathcal{H}_{n}$ with weight function $e^{pt}$. Hence by inequality
(\ref{eqSecondSt}) we obtain
\[
\max_{\zeta\in\left[  t_{j},t_{j+1}\right]  }\left\vert f\left(  \zeta\right)
-f_{0}\left(  \zeta\right)  \right\vert \leq\left(  1+\left\Vert
P^{\mathcal{H}_{n}}\right\Vert _{\text{op}}\right)  \left\Vert f-I_{2}\left(
f\right)  \right\Vert _{\left[  t_{1},t_{n}\right]  }.
\]
Further
\[
\left\Vert f-I_{2}\left(  f\right)  \right\Vert _{\left[  t_{1},t_{n}\right]
}\leq\max_{j=1,...n}M_{\lambda_{0,j}\lambda_{1,j}}^{t_{j}t_{j+1}}\max
_{\zeta\in\left[  t_{j},t_{j+1}\right]  ,j=1,...,n-1}\left\vert L_{\left(
\lambda_{0,j},\lambda_{1,j}\right)  }f\left(  \zeta\right)  \right\vert
\]
This ends the proof.
\end{proof}

\begin{remark}
Note that the last proof also provides an estimate for the derivatives of
second order:
\begin{align*}
&  \max_{t\in\left[  t_{j},t_{j+1}\right]  }\left\vert D_{\lambda_{2,j}%
}D_{\lambda_{3,j}}\left[  F-I_{4}\left(  F\right)  \right]  \left(  t\right)
\right\vert \\
&  \leq\left(  1+\left\Vert P^{\mathcal{H}_{n}}\right\Vert _{\text{op}%
}\right)  \max_{x\in\left[  t_{j},t_{j+1}\right]  ,j=1,...,n-1}\left\vert
L_{\left(  \lambda_{0,j},\lambda_{1,j},\lambda_{2,j},\lambda_{3,j}\right)
}F\left(  x\right)  \right\vert .
\end{align*}

\end{remark}

Now we are able to derive our main result:

\begin{theorem}
Let $\xi_{j}$ for $j=1,...,n-1$ real numbers and $t_{1}<...<t_{n}.$ Let $F\in
C^{4}\left[  t_{1},t_{n}\right]  $ and assume that $I_{4}\left(  F\right)  $
is a piecewise exponential spline for the operators $L_{\left(  \xi_{j}%
,-\xi_{j},\xi_{j},-\xi_{j}\right)  }$ which interpolates $F$ at the points
$t_{1},...,t_{n}$ and
\[
\frac{d}{dt}F\left(  t_{1}\right)  =\frac{d}{dt}I_{4}\left(  F\right)  \left(
t_{1}\right)  \text{ and }\frac{d}{dt}F\left(  t_{n}\right)  =\frac{d}%
{dt}I_{4}\left(  F\right)  \left(  t_{n}\right)  .
\]
Then for any $t\in\left[  t_{1},t_{n}\right]  $ we have the estimate
\[
\left\vert F\left(  t\right)  -I_{4}\left(  F\right)  \left(  t\right)
\right\vert \leq\frac{\Delta^{4}}{16}\frac{5}{4}\max_{\zeta\in\left[
t_{j},t_{j+1}\right]  ,j=1,...,n-1}\left\vert L_{\left(  \xi_{j},-\xi_{j}%
,\xi_{j},-\xi_{j}\right)  }F\left(  \zeta\right)  \right\vert .
\]

\end{theorem}

\begin{proof}
Theorem \ref{ThmMOmega} shows that $M_{\xi_{j},-\xi_{j}}^{t_{j}t_{j+1}}%
\leq\frac{1}{8}\left\vert t_{j+1}-t_{j}\right\vert ^{2}$, and Theorem
\ref{CorLebsgueSym} shows $\left\Vert P^{\mathcal{H}_{n}}\right\Vert \leq4.$
Hence
\[
C=\left(  1+\left\Vert P^{\mathcal{H}_{n}}\right\Vert \right)  \max
_{j=1,...n-1}M_{\xi_{j},-\xi_{j}}^{t_{j}t_{j+1}}\max_{j=1,...n}M_{\xi_{j}%
,-\xi_{j}}^{t_{j}t_{j+1}}\leq\frac{5}{64}\left\vert t_{j+1}-t_{j}\right\vert
^{4}.
\]

\end{proof}

\section{\label{S8}Appendix: Exponential polynomials}

For given complex numbers $\lambda_{0},...,\lambda_{N}$ the elements of the
space
\[
E\left(  \lambda_{0},...,\lambda_{N}\right)  =\left\{  f\in C^{N+1}\left(
\mathbb{R}\right)  :L_{\left(  \lambda_{0},\ldots,\lambda_{N}\right)
}f=0\right\}
\]
are called \emph{exponential polynomials} or $L$\emph{-polynomials, }and
$\lambda_{0},\ldots,\lambda_{N}$ are also called \emph{exponents} or
\emph{frequencies}. In the case of \emph{pairwise different} $\lambda
_{j},j=0,\ldots,N,$ the space$E\left(  \lambda_{0},...,\lambda_{N}\right)  $
is the linear span of the functions%
\[
e^{\lambda_{0}x},e^{\lambda_{1}x},\ldots,e^{\lambda_{N}x}.
\]
When $\lambda_{j}$ occurs $m_{j}$ times in $\Lambda_{N}=\left(  \lambda
_{0},\ldots,\lambda_{N}\right)  ,$ a basis of the space$E\left(  \lambda
_{0},...,\lambda_{N}\right)  $ is given by the linearly independent functions
$x^{s}e^{\lambda_{j}x}$ for $s=0,1,\ldots,m_{j}-1.$

There exists a unique function $\Phi_{\left(  \lambda_{0},...,\lambda
_{N}\right)  }$ in $E\left(  \lambda_{0},...,\lambda_{N}\right)  $ such that
\begin{equation}
\Phi_{\left(  \lambda_{0},...,\lambda_{N}\right)  }\left(  0\right)
=\ldots=\Phi_{\left(  \lambda_{0},...,\lambda_{N}\right)  }^{\left(
N-1\right)  }\left(  0\right)  =0\text{ and }\Phi_{\left(  \lambda
_{0},...,\lambda_{N}\right)  }^{\left(  N\right)  }\left(  0\right)  =1.
\label{eqfundamentfunc}%
\end{equation}
The function $\Phi_{\left(  \lambda_{0},...,\lambda_{N}\right)  }^{\left(
N-1\right)  }\left(  0\right)  $ is called the \emph{fundamental function. }An
explicit definition is the formula
\begin{equation}
\Phi_{\left(  \lambda_{0},...,\lambda_{N}\right)  }\left(  t\right)  =\frac
{1}{2\pi i}\int_{\Gamma_{r}}\frac{e^{tz}}{\left(  z-\lambda_{0}\right)
\cdots\left(  z-\lambda_{N}\right)  }dz, \label{defPhi}%
\end{equation}
where $\Gamma_{r}$ is the path in the complex plane defined by $\Gamma
_{r}\left(  t\right)  =re^{it}$, $t\in\left[  0,2\pi\right]  $, surrounding
all the complex numbers $\lambda_{0},\ldots,\lambda_{N},$ see \cite{micchelli}%
. The integral representation (\ref{defPhi}) implies the formula
\begin{equation}
\left(  \frac{d}{dt}-\lambda_{N+1}\right)  \Phi_{\left(  \lambda_{0}%
,\ldots,\lambda_{N+1}\right)  }\left(  t\right)  =\Phi_{\left(  \lambda
_{0},\ldots,\lambda_{N}\right)  }\left(  t\right)  . \label{rec0}%
\end{equation}
This formula for $\lambda_{N+1}=0$ and the fundamental theorem of calculus
yield:
\begin{equation}
\int_{0}^{h}\Phi_{\left(  \lambda_{0},\ldots,\lambda_{N}\right)  }\left(
t\right)  dt=\Phi_{\left(  \lambda_{0},\ldots,\lambda_{N},0\right)  }\left(
h\right)  \label{eqInt}%
\end{equation}

The following result is a well-known, see e.g. \cite{Sc81},
\cite{KounchevBOOK}.

\begin{proposition}
\label{PropCHebyshev} If $\lambda_{0},...,\lambda_{N}$ are real then the space
$E\left(  \lambda_{0},...,\lambda_{N}\right)  $ is an extended Chebyshev space
on $\mathbb{R}$, i.e. each non-zero function $f\in E\left(  \lambda
_{0},...,\lambda_{N}\right)  $ has at most $N$ zeros (including
multiplicities) on the real line.
\end{proposition}

\begin{remark}
\label{RemarkOMEGA} It is a simple and well-known consequence of the above
Proposition that for every choice of the numbers $t_{1}<...<t_{N}$ and the
data $y_{1},....,y_{N}$ there exists a unique $f\in E\left(  \lambda
_{0},...,\lambda_{N}\right)  $ with $f\left(  t_{j}\right)  =y_{j}$ for
$j=1,...,n$. Hence, for $\left(  \lambda_{0},\lambda_{1},0\right)
\in\mathbb{R}^{3},$ numbers $a<b$ and $t_{\ast}\in\left(  a,b\right)  $ there
exists $f\in E\left(  \lambda_{0},\lambda_{1},0\right)  $ such that $f\left(
a\right)  =f\left(  b\right)  =0$ and $f\left(  t_{\ast}\right)  =1.$ This
implies the existence of the function $\Omega_{\left(  \lambda_{0},\lambda
_{1},0\right)  }^{a,b}$ in section \ref{S4}.
\end{remark}

The fundamental function $\Phi_{\left(  \lambda_{0},...,\lambda_{N}\right)  }$
has a zero of order $N,$ hence it follows that $\Phi_{\left(  \lambda
_{0},...,\lambda_{N}\right)  }\left(  t\right)  \neq0$ for all $t\neq0.$

\begin{proposition}
\label{PropPhipositive}If $\lambda_{0},...,\lambda_{N}$ are real numbers then
$\Phi_{\left(  \lambda_{0},...,\lambda_{N}\right)  }\left(  t\right)  >0$ for
all $t>0.$
\end{proposition}

\begin{lemma}
For real $\lambda_{0},\lambda_{1}$ the following identities hold:
\begin{align}
\int_{0}^{h}\left\vert \Phi_{\left(  \lambda_{0},\lambda_{1}\right)  }\left(
t\right)  \right\vert ^{2}e^{pt}dt  &  =2e^{ph}\Phi_{\left(  2\lambda
_{0},2\lambda_{1},\lambda_{0}+\lambda_{1},-p\right)  }\left(  h\right)
\label{eqId1}\\
\int_{0}^{h}\Phi_{\left(  \lambda_{0},\lambda_{1}\right)  }\left(  t-h\right)
^{2}e^{pt}dt  &  =-2\Phi_{\left(  -2\lambda_{0},-2\lambda_{1},-\lambda
_{0}-\lambda_{1},p\right)  }\left(  h\right) \label{eqId2}\\
\int_{0}^{h}\Phi_{\left(  \lambda_{0},\lambda_{1}\right)  }\left(  t-h\right)
e^{pt}dt  &  =-\Phi_{\left(  -\lambda_{0},-\lambda_{1},p\right)  }\left(
h\right)  \label{eqId3}%
\end{align}

\end{lemma}

\begin{proof}
For $\lambda_{0}\neq\lambda_{1}$ real, the function
\[
f\left(  t\right)  :=\Phi_{\left(  \lambda_{0},\lambda_{1}\right)  }\left(
t\right)  ^{2}=\frac{1}{\left(  \lambda_{1}-\lambda_{0}\right)  ^{2}}\left(
e^{2\lambda_{1}t}-2e^{\left(  \lambda_{1}+\lambda_{0}\right)  t}%
+e^{2\lambda_{0}t}\right)  .
\]
is an exponential polynomial in $E\left(  2\lambda_{1},2\lambda_{2}%
,\lambda_{1}+\lambda_{0}\right)  $ such that $f\left(  0\right)  =f^{\prime
}\left(  0\right)  =0$ and $f^{\prime\prime}\left(  0\right)  =2.$ Thus
\[
\Phi_{\left(  \lambda_{0},\lambda_{1}\right)  }\left(  t\right)  ^{2}%
=2\Phi_{\left(  2\lambda_{0},2\lambda_{1},\lambda_{0}+\lambda_{1}\right)
}\left(  t\right)  .
\]
This formula is also valid for $\lambda_{0}=\lambda_{1}$ by a limit argument
in $\left(  \lambda_{0},\lambda_{1}\right)  $. With (\ref{eqTrans}) we see
that
\[
\left\vert \Phi_{\left(  \lambda_{0},\lambda_{1}\right)  }\left(  t\right)
\right\vert ^{2}e^{pt}=2e^{pt}\Phi_{\left(  2\lambda_{0},2\lambda_{1}%
,\lambda_{0}+\lambda_{1}\right)  }\left(  t\right)  =2\Phi_{\left(
p+2\lambda_{0},p+2\lambda_{1},p+\lambda_{0}+\lambda_{1}\right)  }\left(
t\right)
\]
Integration as in (\ref{eqInt}) shows that
\[
\int_{0}^{h}\left\vert \Phi_{\left(  \lambda_{0},\lambda_{1}\right)  }\left(
t\right)  \right\vert ^{2}e^{pt}dt=2\Phi_{\left(  p+2\lambda_{0}%
,p+2\lambda_{1},p+\lambda_{0}+\lambda_{1},0\right)  }\left(  h\right)
\]
which with (\ref{eqTrans}) gives (\ref{eqId1}).

Note that substitution $\tau_{1}=h-t$ yields
\[
I:=\int_{0}^{h}\Phi_{\left(  \lambda_{0},\lambda_{1}\right)  }\left(
t-h\right)  ^{2}e^{pt}dt=e^{ph}\int_{0}^{h}\Phi_{\left(  \lambda_{0}%
,\lambda_{1}\right)  }\left(  -\tau\right)  ^{2}e^{-pt}d\tau.
\]
By (\ref{eqMinus}), $\Phi_{\left(  \lambda_{0},\lambda_{1}\right)  }\left(
-\tau\right)  =-\Phi_{\left(  -\lambda_{0},-\lambda_{1}\right)  }\left(
\tau\right)  .$ Using (\ref{eqId1}) for $\left(  -\lambda_{0},-\lambda
_{1}\right)  $ and $-p$ shows that
\[
I=e^{ph}2\Phi_{\left(  -p-2\lambda_{0},-p-2\lambda_{1},-p-\lambda_{0}%
-\lambda_{1},0\right)  }\left(  h\right)  =2\Phi_{\left(  -2\lambda
_{0},-2\lambda_{1},-\lambda_{0}-\lambda_{1},p\right)  }\left(  h\right)  ,
\]
which proves (\ref{eqId2}). The case (\ref{eqId3}) is similar and left to the reader.
\end{proof}

Let $g:\left[  0,b\right]  \rightarrow\mathbb{C}$ be differentiable and
$f:\left[  0,b\right]  \rightarrow\mathbb{C}$ continuous. Then it is well
known that
\[
A\left(  y\right)  =\int_{0}^{y}f\left(  t\right)  g\left(  y-t\right)  dt
\]
is differentiable and
\[
\frac{d}{dy}\int_{0}^{y}f\left(  t\right)  g\left(  y-t\right)  dt=f\left(
y\right)  g\left(  0\right)  +\int_{0}^{y}f\left(  t\right)  \left(  \frac
{d}{dy}g\right)  \left(  y-t\right)  dy.
\]
For the differential operator $D_{\lambda}=\frac{d}{dy}-\lambda$ it is
straightforward to verify that
\[
D_{\lambda}\int_{0}^{y}f\left(  t\right)  g\left(  y-t\right)  dt=f\left(
y\right)  g\left(  0\right)  +\int_{0}^{y}f\left(  t\right)  D_{\lambda
}g\left(  y-t\right)  dy
\]
The next result follows by induction.

\begin{proposition}
\label{PropRecInt2}Let $g:\left[  0,b\right]  \rightarrow\mathbb{C}$ be $k+1$
times continuously differentiable function with $g^{\left(  l\right)  }\left(
0\right)  =0$ for $l=0,...,k$ and assume that $f:\left[  0,b\right]
\rightarrow\mathbb{C}$ is continuous. Then
\[
A\left(  y\right)  =\int_{0}^{y}f\left(  t\right)  g\left(  y-t\right)  dt
\]
is $k+1$ times differentiable and
\[
D_{\lambda_{k}}\cdots D_{\lambda_{0}}A\left(  y\right)  =\int_{0}^{y}f\left(
t\right)  \left(  D_{\lambda_{k}}\cdots D_{\lambda_{0}}g\right)  \left(
y-t\right)  dy.
\]

\end{proposition}

\begin{theorem}
\label{ThmApp1}Let $\lambda_{0},...,\lambda_{n},\lambda_{n+1},...,\lambda
_{n+m}$ be complex numbers. Then the following identity holds:
\begin{equation}
\int_{0}^{y}\Phi_{\left(  \lambda_{0},...,\lambda_{n}\right)  }\left(
t\right)  \Phi_{\left(  \lambda_{n+1},...,\lambda_{n+m}\right)  }\left(
y-t\right)  dt=\Phi_{\left(  \lambda_{0},...,\lambda_{n+m}\right)  }\left(
y\right)  . \label{eqgrand}%
\end{equation}

\end{theorem}

\begin{proof}
Let us denote the integral by $A\left(  y\right)  $. Define $f\left(
t\right)  =\Phi_{\left(  \lambda_{0},...,\lambda_{n}\right)  }\left(
t\right)  $ and $g\left(  y\right)  =\Phi_{\left(  \lambda_{n+1}%
,...,\lambda_{n+m}\right)  }\left(  y\right)  .$Then $g^{\left(  l\right)
}\left(  0\right)  =0$ for $l=0,....,m-1$ and $g^{\left(  m\right)  }\left(
0\right)  =1.$ Proposition \ref{PropRecInt2} (for $k=m-1)$ shows that
\[
D_{\lambda_{n+l}}....D_{\lambda_{n+m}}A\left(  y\right)  =\int_{0}^{y}f\left(
t\right)  D_{\lambda_{n+l}}....D_{\lambda_{n+m}}\Phi_{\left(  \lambda
_{n+1},...,\lambda_{n+m}\right)  }\left(  y-t\right)  dt
\]
for each $l=2,....,m$, and we conclude that $A^{\left(  k\right)  }\left(
0\right)  =0$ for $k=0,...,m-1.$ For $l=2$ we obtain
\[
G\left(  y\right)  :=D_{\lambda_{n+2}}....D_{\lambda_{n+m}}A\left(  y\right)
=\int_{0}^{y}f\left(  t\right)  \Phi_{\left(  \lambda_{n+1}\right)  }\left(
y-t\right)  dt.
\]
Proposition \ref{PropRecInt2} applied to $g_{2}\left(  y\right)
=\Phi_{\left(  \lambda_{n+1}\right)  }\left(  y\right)  $ shows that
\[
D_{\lambda_{n+1}}G\left(  y\right)  =f\left(  y\right)  g_{2}\left(  0\right)
+\int_{0}^{y}f\left(  t\right)  D_{\lambda_{n+1}}g_{2}\left(  y-t\right)  dy.
\]
The last integral vanishes since $D_{\lambda_{n+1}}g_{2}=0$. Since
$g_{2}\left(  0\right)  =1$ and $f\left(  t\right)  =\Phi_{\left(  \lambda
_{0},...,\lambda_{n}\right)  }\left(  t\right)  $ we obtain
\[
D_{\lambda_{n+1}}....D_{\lambda_{n+m}}A\left(  y\right)  =\Phi_{\left(
\lambda_{0},...,\lambda_{n}\right)  }\left(  y\right)  .
\]
Thus $A^{\left(  k\right)  }\left(  0\right)  =0$ for all $k=0,....,n+m,$ and
$A^{\left(  n+m\right)  }\left(  0\right)  =1.$ In order to show that
$A\left(  y\right)  =\Phi_{\left(  \lambda_{0},...,\lambda_{n+m}\right)
}\left(  y\right)  $ it suffices to show that it an exponential polynomial
with frequencies $\lambda_{0},...,\lambda_{n+m},$ but this is clear since
\[
D_{\lambda_{0}}....D_{\lambda_{n}}D_{\lambda_{n+1}}....D_{\lambda_{n+m}%
}A\left(  y\right)  =D_{\lambda_{0}}....D_{\lambda_{n}}\Phi_{\left(
\lambda_{0},...,\lambda_{n}\right)  }\left(  y\right)  =0.
\]

\end{proof}

\begin{lemma}
\label{Lem1App}Let $a$ be negative number and $b$ a real number. Then the
function $F_{\left(  a,b\right)  }$ defined by
\[
F_{\left(  a,b\right)  }\left(  t\right)  =2C_{a,b}\Phi_{\left(
a-1,1-a,0,b\right)  }\left(  t\right)  -\Phi_{\left(  a,1,-a,-1\right)
}\left(  t\right)
\]
is positive for all $t>0$, where we have put
\begin{equation}
C_{a,b}:=\max\left\{  \frac{1}{2},\frac{2-a-b}{2\left(  2-a\right)  }%
,\frac{2a^{2}+\left(  2a-1\right)  b+2-3a}{2\left(  1-2a\right)  \left(
2-a\right)  },\frac{1-b}{2\left(  1-2a\right)  }\right\}  . \label{eqCondC}%
\end{equation}

\end{lemma}

\begin{proof}
In view of Lemma \ref{LemPos1} and the fact that $F_{\left(  a,b\right)
}\left(  0\right)  =F_{\left(  a,b\right)  }^{\prime}\left(  0\right)  =0$ it
suffices to show that for $C=C_{a,b}$
\[
G:=D_{b}D_{0}F_{\left(  a,b\right)  }=2C\Phi_{\left(  a-1,1-a\right)  }%
-\Phi_{\left(  a,1,-a,-1\right)  }^{\prime\prime}+b\Phi_{\left(
a,1,-a,-1\right)  }^{\prime}%
\]
is strictly positive for $t>0.$ If $a^{2}\neq1$ and $a\neq0$ we know that
\begin{align*}
\Phi_{\left(  a,-a,1,-1\right)  }\left(  t\right)   &  =\frac{1}{1-a^{2}%
}\left(  \sinh t-\frac{1}{a}\sinh at\right) \\
\Phi_{\left(  a,-a,1,-1\right)  }^{\prime}\left(  t\right)   &  =\frac
{1}{1-a^{2}}\left(  \cosh t-\cosh at\right) \\
\Phi_{\left(  a,-a,1,-1\right)  }^{\prime\prime}\left(  t\right)   &
=\frac{1}{1-a^{2}}\left(  \sinh t-a\sinh at\right) \\
\Phi_{\left(  a-1,1-a\right)  }\left(  t\right)   &  =\frac{\sinh\left(
1-a\right)  t}{1-a}=\frac{1}{2}\frac{e^{t}e^{-at}-e^{-t}e^{at}}{\left(
1-a\right)  }.
\end{align*}
Hence
\[
\left(  1-a\right)  G\left(  t\right)  =2C\sinh\left(  1-a\right)
t-\frac{\sinh t-a\sinh at}{1+a}+b\frac{\cosh t-\cosh at}{1+a}.
\]
Let us put $v=e^{-at}$ for $t\geq0.$ Then $v\geq1$ since $a<0$ and
$v^{\frac{1}{-a}}=\left(  e^{-at}\right)  ^{\frac{1}{-a}}=e^{t}.$ It follows
that
\begin{align*}
&  2\left(  1-a\right)  G\left(  t\right) \\
&  =2C\left(  v^{1-\frac{1}{a}}-v^{-1+\frac{1}{a}}\right)  -\frac{\left(
v^{\frac{1}{-a}}-v^{\frac{1}{a}}\right)  -a\left(  v^{-1}-v\right)  }%
{1+a}+b\frac{v^{\frac{1}{-a}}+v^{\frac{1}{a}}-v^{-1}-v}{1+a}\\
&  =2C\left(  v^{1-\frac{1}{a}}-v^{-1+\frac{1}{a}}\right)  +\frac{b-1}%
{a+1}v^{\frac{1}{-a}}+\frac{b+1}{a+1}v^{\frac{1}{a}}-\frac{b-a}{a+1}%
v^{-1}-\frac{b+a}{a+1}v.
\end{align*}
Multiply this expression by $v^{1-\frac{1}{a}}$. Then it suffices to show that
for all $v>1$
\[
F\left(  v\right)  :=2Cv^{2-\frac{2}{a}}-2C+\frac{b-1}{a+1}v^{1-\frac{2}{a}%
}+\frac{b+1}{a+1}v-\frac{b-a}{a+1}v^{-\frac{1}{\alpha}}-\frac{b+a}%
{a+1}v^{2-\frac{1}{\alpha}}%
\]
is strictly positive. Note that $F\left(  1\right)  =0.$ By Lemma
\ref{LemPos1} it suffices to show that
\begin{align*}
F^{\prime}\left(  v\right)   &  =2C\left(  2-\frac{2}{a}\right)  v^{1-\frac
{2}{a}}+\frac{\left(  b-1\right)  \left(  1-\frac{2}{a}\right)  v^{-\frac
{2}{a}}}{a+1}+\frac{b+1}{a+1}+\frac{\left(  b-a\right)  v^{-1-\frac{1}{a}}%
}{a\left(  a+1\right)  }\\
&  -\frac{\left(  b+a\right)  \left(  2-\frac{1}{a}\right)  v^{1-\frac{1}{a}}%
}{a+1}%
\end{align*}
is strictly positive. Note that
\[
F^{\prime}\left(  1\right)  =\frac{2}{a}\left(  2C-1\right)  \left(
a-1\right)  \geq0
\]
for any $a<0$ since $C\geq\frac{1}{2}.$ Again by Lemma \ref{LemPos1}, it
suffices to show that
\begin{align*}
F^{\prime\prime}\left(  v\right)   &  =2C\left(  2-\frac{2}{a}\right)  \left(
1-\frac{2}{a}\right)  v^{-\frac{2}{a}}-\frac{\left(  b-1\right)  }{a+1}%
\frac{2}{a}\left(  1-\frac{2}{a}\right)  v^{-1-\frac{2}{a}}\\
&  +\frac{b-a}{a+1}\left(  -1-\frac{1}{a}\right)  \frac{1}{a}v^{-2-\frac{1}%
{a}}-\frac{b+a}{a+1}\left(  1-\frac{1}{a}\right)  \left(  2-\frac{1}%
{a}\right)  v^{-\frac{1}{a}}%
\end{align*}
is strictly positive. Multiply the expression with $v^{2+\frac{1}{a}}$ and
define
\begin{align*}
\widetilde{F}\left(  v\right)   &  =v^{2+\frac{1}{a}}F^{\prime\prime}\left(
v\right)  =2C\left(  2-\frac{2}{a}\right)  \left(  1-\frac{2}{a}\right)
v^{2-\frac{1}{a}}-\frac{\left(  b-1\right)  }{a+1}\frac{2}{a}\left(
1-\frac{2}{a}\right)  v^{1-\frac{1}{a}}\\
&  +\frac{b-a}{a+1}\left(  -1-\frac{1}{a}\right)  \frac{1}{a}-\frac{b+a}%
{a+1}\left(  1-\frac{1}{a}\right)  \left(  2-\frac{1}{a}\right)  v^{2}.
\end{align*}
Note that
\[
\widetilde{F}\left(  1\right)  =-\frac{2}{a^{2}}\left(  a-1\right)  \left(
a+b+4C-2aC-2\right)  .
\]
Since $a<0$ we infer that $\widetilde{F}\left(  1\right)  \geq0$ when
$a+b+4C_{a}-2aC_{a}-2\geq0.$ This is true since by our assumption
\[
C\geq\frac{2-a-b}{2\left(  2-a\right)  }.
\]
By Lemma \ref{LemPos1} it suffices to show that $\widetilde{F}^{\prime}\left(
v\right)  $ is positive for $v>1$ where
\begin{align*}
\widetilde{F}^{\prime}\left(  v\right)   &  =4C\left(  1-\frac{1}{a}\right)
\left(  1-\frac{2}{a}\right)  \left(  2-\frac{1}{a}\right)  v^{1-\frac{1}{a}%
}-\frac{b-1}{a+1}\frac{2}{a}\left(  1-\frac{2}{a}\right)  \left(  1-\frac
{1}{a}\right)  v^{-\frac{1}{a}}\\
&  -\frac{b+a}{a+1}\left(  1-\frac{1}{a}\right)  \left(  2-\frac{1}{a}\right)
2v.
\end{align*}
Divide by $v\geq1$ and $1-\frac{1}{a}\geq0,$ then it suffices to show that
\begin{align*}
H\left(  v\right)   &  =\frac{\widetilde{F}^{\prime}\left(  v\right)
}{v\left(  1-\frac{1}{a}\right)  }\\
&  =4C\left(  1-\frac{2}{a}\right)  \left(  2-\frac{1}{a}\right)  v^{-\frac
{1}{a}}-\frac{b-1}{a+1}\frac{2}{a}\left(  1-\frac{2}{a}\right)  v^{-1-\frac
{1}{a}}-\frac{b+a}{a+1}\left(  2-\frac{1}{a}\right)  2
\end{align*}
is strictly positive for $v>1.$ A computation shows that
\[
\frac{a^{2}}{2}H\left(  1\right)  =2C\left(  1-2a\right)  \left(  2-a\right)
+3a+2b-2ab-2a^{2}-2.
\]
We see that $H\left(  1\right)  \geq0$ if we assume the inequality
\begin{equation}
C\geq\frac{2a^{2}+\left(  2a-1\right)  b+2-3a}{2\left(  1-2a\right)  \left(
2-a\right)  }. \label{eqCondC2}%
\end{equation}
By Lemma \ref{LemPos1} it suffices to show $H^{\prime}$ is strictly positive
for $v>1.$ This is seen from the following argument:
\[
H^{\prime}\left(  v\right)  =4C\left(  1-\frac{2}{a}\right)  \left(
2-\frac{1}{a}\right)  \left(  -\frac{1}{a}\right)  v^{-1-\frac{1}{a}}%
-\frac{b-1}{a+1}\frac{2}{a}\left(  1-\frac{2}{a}\right)  \left(  -1-\frac
{1}{a}\right)  v^{-2-\frac{1}{a}}%
\]
is positive if for all $v>1$
\[
Cv>\frac{b-1}{a+1}\frac{\frac{2}{a}\left(  1-\frac{2}{a}\right)  \left(
-1-\frac{1}{a}\right)  }{4\left(  1-\frac{2}{a}\right)  \left(  2-\frac{1}%
{a}\right)  \left(  -\frac{1}{a}\right)  }=\frac{b-1}{4a-2}=\frac{1-b}{2-4a}.
\]

\end{proof}

\end{document}